\documentclass[journal]{IEEEtran}
\usepackage{graphicx}
\usepackage{graphics} % for pdf, bitmapped graphics files
\usepackage{amsmath} % assumes amsmath package installed
\usepackage{amssymb}  % assumes amsmath package installed
\usepackage{amsfonts}
\usepackage[compress]{cite}
\usepackage{tikz}
\usetikzlibrary{matrix,arrows}
\usepackage{algorithmic}
\usepackage{algorithm}

\usepackage{enumitem}

\newtheorem{theorem}{\bf Theorem}

\newtheorem{remark}{\bf Remark}
\newtheorem{definition}{\bf Definition}

\newtheorem{proposition}{\bf Proposition}
\newtheorem{lemma}{\bf Lemma}

\def\QED{~\rule[-1pt]{5pt}{5pt}\par\medskip}
\newenvironment{proof}{{\bf Proof: \ }}{ \hfill \QED}

\newcommand{\bx}{{\bf x}}
\newcommand{\by}{{\bf y}}
\newcommand{\bz}{{\bf z}}
\newcommand{\bv}{{\bf v}}
\newcommand{\bw}{{\bf w}}
\newcommand{\bg}{{\bf g}}

\newcommand{\upd}{}

\DeclareMathOperator{\diag}{diag}

\newcommand*{\SHORTVERSION}{}%  %To be submitted

\newcommand*{\FIGDOUBLECOL}{}%  

\begin{document}
%
% paper title
% can use linebreaks \\ within to get better formatting as desired
% Do not put math or special symbols in the title.
\title{Semidefinite Programming Approach to Gaussian Sequential Rate-Distortion Trade-offs}
%
%
% author names and IEEE memberships
% note positions of commas and nonbreaking spaces ( ~ ) LaTeX will not break
% a structure at a ~ so this keeps an author's name from being broken across
% two lines.
% use \thanks{} to gain access to the first footnote area
% a separate \thanks must be used for each paragraph as LaTeX2e's \thanks
% was not built to handle multiple paragraphs
%

\author{Takashi Tanaka, %~\IEEEmembership{Member,~IEEE,}
        Kwang-Ki K. Kim, %~\IEEEmembership{Member,~IEEE,}
        Pablo A. Parrilo, %~\IEEEmembership{Fellow,~OSA,} 
        and~Sanjoy K. Mitter%~\IEEEmembership{Life~Fellow,~IEEE}% <-this % stops a space
\thanks{T.~Tanaka is with ACCESS Linnaeus Center, KTH Royal Institute of Technology, Stockholm, 10044  Sweden.}
\thanks{P.~A.~Parrilo, and S.~K.~Mitter are with the Laboratory for Information and Decision Systems, Massachusetts Institute of Technology, Cambridge,
MA, 02139 USA.}% <-this % stops a space
\thanks{K.-K.~K.~Kim is with the Electronic Control Development Team of Hyundai Motor Company (HMC) Research \& Development Division in South Korea. }% <-this % stops a space
}

\maketitle

% As a general rule, do not put math, special symbols or citations
% in the abstract or keywords.
\begin{abstract}
Sequential rate-distortion (SRD) theory provides a framework for studying the fundamental trade-off between data-rate and data-quality in real-time communication systems. 
In this paper, we consider the SRD problem for multi-dimensional time-varying Gauss-Markov processes under mean-square distortion criteria.
We first revisit the sensor-estimator separation principle, which asserts that considered SRD problem is equivalent to a joint sensor and estimator design problem in which data-rate of the sensor output is minimized while the estimator's performance satisfies the distortion criteria.
We then show that the optimal joint design can be performed by semidefinite programming. 
A semidefinite representation of the corresponding SRD function is obtained.
Implications of the obtained result in the context of zero-delay source coding theory and applications to networked control theory are also discussed.
\end{abstract}

% Note that keywords are not normally used for peerreview papers.

% For peer review papers, you can put extra information on the cover
% page as needed:
% \ifCLASSOPTIONpeerreview
% \begin{center} \bfseries EDICS Category: 3-BBND \end{center}
% \fi
%
% For peerreview papers, this IEEEtran command inserts a page break and
% creates the second title. It will be ignored for other modes.
\IEEEpeerreviewmaketitle

\section{Introduction}

{\upd 
In this paper, we study a fundamental performance limitation of zero-delay communication systems using the sequential rate-distortion (SRD) theory.
Suppose that  $\bx_t$ is an $\mathbb{R}^n$-valued discrete time random process with known statistical properties. At every time step, the encoder observes a realization of the source $\bx_t$ and generates a binary sequence ${\bf{b}}_t\in\{0,1\}^{l_t}$ of length $l_t$, which is transmitted to the decoder. The decoder produces an estimation $\bz_t$ of $\bx_t$ based on the messages $\bf{b}_t$ received up to time $t$. Both encoder and decoder have infinite memories of the past. 
A zero-delay communication system is determined by a selected encoder-decoder pair, whose performance is analyzed  in the trade-off between the rate (\emph{viz}. the average number of \emph{bits} that must be transmitted per time step) and the distortion (\emph{viz}. the discrepancy between the source signal  $\bx_t$ and the reproduced signal $\bz_t$). The region in the rate-distortion plane achievable by a zero-delay communication system is referred to as the \emph{zero-delay rate-distortion region}.\footnote{Formal definition of the zero-delay rate-distortion region is given in Section~\ref{secapplicationszerodelay}.}

The \emph{standard rate-distortion region} identified by Shannon only provides a conservative outer bound of the zero-delay rate-distortion region.
This is because, in general, achieving the standard rate-distortion region requires the use of anticipative (non-causal) codes (e.g., \cite[Theorem 10.2.1]{CoverThomas}). 
It is well known that the standard rate-distortion region can be expressed by  the \emph{rate-distortion function}\footnote{This quantity is defined by the infimum of the mutual information between the source and the reproduction subject to the distortion constraint \cite[Theorem 10.2.1]{CoverThomas}.} for general sources. 
In contrast, description of the zero-delay rate-distortion region requires more case-dependent knowledge of the optimal source coding schemes.
 For scalar memoryless sources, it is shown that the optimal performance of zero-delay codes is achievable by a scalar quantizer \cite{gaarder1982optimal}. Witsenhausen \cite{witsenhausen1979structure} showed that for the $k$-th order Markov sources, there exists an optimal zero-delay quantizer with memory structure of order $k$.
Neuhoff and Gilbert considered entropy-coded quantizers within the class of \emph{causal source codes} \cite{neuhoff1982causal}, and showed that for memoryless sources, the optimal performance is achievable by time-sharing memoryless codes. This result is extended to sources with memory in \cite{linder2006causal}.  An optimal memory structure of zero-delay quantizers for partially observable Markov processes on abstract (Polish) spaces is identified in \cite{yuksel2013optimal}.
The rate of finite-delay source codes for general sources and general distortion measures is analyzed in \cite{kostina2012fixed}.
Zero-delay or finite-delay joint source-channel coding problems have also been studied in the literature; \cite{walrand1983optimal,teneketzis2006structure,mahajan2009optimal,Gorantla2011information} to name a few.

In \cite{TatikondaThesis,tatikonda2004}, Tatikonda et al. studied the zero-delay rate-distortion region using a quantity called \emph{sequential rate-distortion function},\footnote{Closely related or apparently equivalent notions to the sequential rate-distortion function  have been given various names in the literature, including nonanticipatory $\epsilon$-entropy \cite{gorbunov1973nonanticipatory}, constrained distortion rate function \cite{bucy1980distortion}, causal rate-distortion function \cite{derpich2012}, and nonanticipative rate-distortion function \cite{charalambous2014nonanticipative}.} which is defined as the infimum of the Massey's directed information  \cite{massey1990causality} from the source process to the reproduction process subject to the distortion constraint.
Although the SRD function does not coincide with the boundary of the zero-delay rate-distortion region in general, 
it is recently shown that the SRD function provides a tight outer bound of the zero-delay rate-distortion region achievable by uniquely decodable codes \cite{silva2011,derpich2012}. 
This observation shows an intimate connection between the SRD function and the fundamental performance limitations of real-time communication systems. 
For this reason, we consider the SRD function as the main object of interest in this paper.

Closely related quantity to the SRD function was studied by Gorbunov and Pinsker \cite{gorbunov1973nonanticipatory} in the early 1970's.
Bucy \cite{bucy1980distortion} derived the SRD function for Gauss-Markov processes in a simple case.
In his approach, the problem of deriving the SRD function for Gauss-Markov processes under mean-square distortion criteria  (which henceforth will be simply referred to as the \emph{Gaussian SRD problem}) is viewed as a sensor-estimator joint design problem to minimize the estimation error subject to the data-rate constraint. 
This approach is justified by the ``sensor-estimator separation principle," which  
asserts that  an optimal solution (i.e., the optimal stochastic kernel, to be made precise in the sequel) to the Gaussian SRD problem is realizable by a two-stage mechanism with a linear-Gaussian memoryless sensor and the Kalman filter. 
Although this fact is implicitly shown in \cite{TatikondaThesis,tatikonda2004}, for completeness, we reproduce a proof  in this paper based on a technique used in  \cite{TatikondaThesis,tatikonda2004}.

The sensor-estimator separation principle gives us a structural understanding of the Gaussian SRD problem.
In particular, based on this principle,  we show that the Gaussian SRD problem can be formulated as a semidefinite programming problem (Theorem~\ref{propsdp}), which is the main contribution of this paper.
We derive a computationally accessible form (namely a semidefinite representation\footnote{To be precise, we show that the \emph{exponentiated} SRD function for multidimensional Gauss-Markov source is semidefinite representable by (\ref{stationarylmi}).} \cite{blekherman2013semidefinite}) of the SRD function, and provide an efficient algorithm to solve Gaussian SRD problems numerically. 

The semidefinite representation of the SRD function may be compared with an alternative analytical approach via Duncan's theorem, which states that ``twice the mutual information is merely the integration of the trace of the optimal mean square filtering error" \cite{duncan1970calculation}. 
Duncan's result was significantly generalized as the ``I-MMSE" relationships in non-causal \cite{guo2005mutual} and causal \cite{weissman2013directed} estimation problems.
Our SDP-based approaches are applicable to the cases with multi-dimensional and time-varying Gauss-Markov sources to which the existing I-MMSE formulas cannot be applied straightforwardly.
Although we focus on the Gaussian SRD problems in this paper, we note that the standard RD and SRD problems for general sources and distortion measures in abstract (Polish) spaces are discussed in  \cite{rezaei2006rate} and\cite{charalambous2014nonanticipative}, respectively.

}

This paper is organized as follows. In Section \ref{secformulation}, we formally introduce the Gaussian SRD problem, which is the main problem considered in this paper. In Section  \ref{secequivalence}, we show that the Gaussian SRD problem is equivalent to what we call the linear-Gaussian sensor design problem, which formally establishes the sensor-estimator separation principle.  Then, in Section \ref{secsdp}, we show that the linear-Gaussian sensor design problem can be reduced to an SDP problem, which thus provides us an SDP-based solution synthesis procedure for Gaussian SRD problems. Extensions to stationary and infinite horizon problems are given in Section \ref{secstationary}. In Section \ref{secapplications}, we consider applications of SRD theory to real-time communication systems and networked control systems.
Simple simulation results will be presented in Section \ref{secexample}. We conclude in Section \ref{secconclusion}.

{\bf Notation:} Let $\mathcal{X}$ be an Euclidean space, and $ \mathcal{B_X}$ be the Borel $\sigma$-algebra on $\mathcal{X}$. 
Let $(\Omega, \mathcal{F}, \mathcal{P})$ be a probability space, and $\bx: (\Omega, \mathcal{F})\rightarrow (\mathcal{X}, \mathcal{B_X})$ be a random variable. Throughout the paper, we use lower case boldface symbols such as $\bx$ to denote random variables, while $x\in \mathcal{X}$ is a realization of $\bx$.
We denote by $q_\bx$ the probability measure of $\bx$ defined by $q_\bx(A)=\mathcal{P}(\{\omega: \bx(\omega)\in A\})$ for every $A\in \mathcal{B_X}$. When no confusion occurs, this measure will be also denoted by $q_\bx(x)$ or $q(x)$.
For a Borel measurable function $f: \mathcal{X}\rightarrow \mathbb{R}$, we write $\mathbb{E}f(\bx)\triangleq \int f(x)q_\bx(dx)$.
For a random vector, we write $\bx^t\triangleq(\bx_0,\cdots, \bx_t)$ or $\bx^t\triangleq(\bx_1,\cdots, \bx_t)$ depending on the initial index, and $\bx_s^t\triangleq(\bx_s,\cdots, \bx_t)$.
Let $\Theta$ be a real symmetric matrix of size $n\times n$. Notations $\Theta \succ 0$ or $\Theta \in \mathbb{S}_{++}^n$ (resp. $\Theta \succeq 0$ or $\Theta \in \mathbb{S}_{+}^n$) mean that  $\Theta$ is a positive definite (resp. positive semidefinite) matrix. For a positive semidefinite matrix $\Theta$, we write $\|x\|_\Theta \triangleq \sqrt{x^\top \Theta x}$.

\section{Problem Formulation}
\label{secformulation}
We begin our discussion with an estimation-theoretic interpretation of a simple rate-distortion trade-off problem. Recall that a rate-distortion problem for a scalar Gaussian random variable $\bx\sim\mathcal{N}(0,1)$ with the mean square distortion constraint is an optimization problem of the following form:
\begin{align}
\min & \;\; I(\bx;\bz) \label{scalarrd}\\
\text{s.t. } & \;\; \mathbb{E}(\bx-\bz)^2 \leq D. \nonumber
\end{align}
Here, $\bz$ is a reproduction of the source $\bx$, and $I(\bx;\bz)$ denotes the mutual information between $\bx$ and $\bz$.
{\upd The minimization is over the space of reproduction policies, i.e., stochastic kernels $q(dz|x)$. 
The optimal value of (\ref{scalarrd}) is known as the rate-distortion function, $R(D)$, and can be explicitly obtained \cite{CoverThomas} as}
\[
R(D)=\max \left\{0, \frac{1}{2}\log\left(\frac{1}{D}\right)\right\}.
\]
{\upd It is also possible to write the optimal reproduction policy $q(dz|x)$  explicitly. To this end, consider a linear sensor} 
\begin{equation}
\label{scalarchannel}
\by=c\bx+\bv
\end{equation}
where $\bv\sim\mathcal{N}(0, \sigma^2)$ is a Gaussian noise independent of $\bx$.
{\upd Also, let 
\begin{equation}
\label{scalarmmse}
\bz=\mathbb{E}(\bx|\by)
\end{equation}
be the least mean square error estimator of $\bx$ given $\by$.
Notice that the right hand side of (\ref{scalarmmse}) is given by $\frac{c}{c^2+\sigma^2}\by$.
Then, it can be shown that an optimal solution $q(dz|x)$ to (\ref{scalarrd}) is a composition of (\ref{scalarchannel}) and (\ref{scalarmmse}), provided that the signal-to-noise ratio of the sensor (\ref{scalarchannel}) is chosen to be}
\begin{equation}
\label{eqscalarsnr}
\mathsf{SNR}\triangleq \frac{c^2}{\sigma^2}=\max \left\{0,\frac{1}{D}-1\right\}.
\end{equation}
{\upd This gives us the following notable observations:}
\begin{itemize}[leftmargin=3ex]
\item {\bf Fact 1:} A ``sensor-estimator separation principle'' holds for the Gaussian rate-distortion problem  (\ref{scalarrd}), in the sense that {\upd an optimal reproduction policy $q(dz|x)$ can be written as a two-stage mechanism with a linear sensor mechanism (\ref{scalarchannel}) and a least mean square error estimator (\ref{scalarmmse}).}
\item {\bf Fact 2:} The original infinite dimensional optimization problem (\ref{scalarrd}) with respect to $q(dz|x)$ is reduced to a simple optimization problem in terms of a scalar parameter $\textsf{SNR}$. Moreover, for a given  $D>0$, the optimal choice of $\textsf{SNR}$ is given by a closed-form expression (\ref{eqscalarsnr}).
\end{itemize}
These facts can be significantly generalized, and serve as a guideline to develop a solution synthesis for Gaussian SRD problems in this paper.

\subsection{Gaussian SRD problem}
\label{secpsrdformulation}
The Gaussian SRD problem can be viewed as a generalization of (\ref{scalarrd}). Let $\{\bx_t\}$ be an $\mathbb{R}^{n_t}$-valued Gauss-Markov process
\begin{equation}
\label{gmprocess}
\bx_{t+1}=A_t \bx_t+\bw_t, \;\; t=0, 1, \cdots, T-1
\end{equation}
where {\upd  $\bx_0\sim\mathcal{N}(0, P_0), P_0\succ 0$ and $\bw_t\sim \mathcal{N}(0, W_t), W_t \succ 0$  for $t=0,1,\cdots, T-1$ are mutually independet Gaussian random variables.}
The Gaussian SRD problem is formulated as
\begin{subequations}
\begin{align}
\text{{\bf (P-SRD): }}\hspace{1ex} 
\min_{\gamma\in\Gamma} & \;\; I(\bx^T \rightarrow \bz^T) \label{psrd1}\\
\text{s.t. } & \;\;  \mathbb{E}\|\bx_t-\bz_t\|_{\Theta_t}^2 \leq D_t \label{psrd2}
\end{align}
\end{subequations}
where (\ref{psrd2}) is imposed for every $t = 1,\cdots,T$.
Here, $\{\bz_t\}$ is an $\mathbb{R}^{n_t}$-valued reproduction of $\{\bx_t\}$.
The minimization (\ref{psrd1}) is over the space $\Gamma$ of zero-delay reproduction policies of $\bz_t$ given $\bx^t$ and $\bz^{t-1}$, i.e., the sequences of causal stochastic kernels\footnote{See Appendix \ref{appmathprelim} for a formal description of causal stochastic kernels.} $\gamma=\otimes_{t=1}^T q(dz_t|x^t,z^{t-1})$. 
The term $I(\bx^T\rightarrow \bz^T)$ is known as \emph{directed information}, 
introduced by Massey \cite{massey1990causality} following Marko's earlier work \cite{marko1973bidirectional}, and is defined by
\begin{equation}
\label{eqdirectedinfo}
I(\bx^T\rightarrow \bz^T)\triangleq \sum_{t=1}^T I(\bx^t;\bz_t|\bz^{t-1}).
\end{equation}
 The Gaussian SRD problem is visualized in Fig. \ref{fig:SRD}.
\begin{remark}
\label{remark_directedinfo}
Directed information measures the amount of information flow from $\{\bx_t\}$ to $\{\bz_t\}$ and is not symmetric, i.e., $I(\bx^T\rightarrow \bz^T)\neq I(\bz^T\rightarrow \bx^T)$ in general. However, when the process $\{\bz_t\}$ is causally dependent on $\{\bx_t\}$ and $\{\bx_t\}$ is not affected by $\{\bz_t\}$, it can be shown \cite{massey2005conservation} that
$I(\bx^T\rightarrow \bz^T)=I(\bx^T; \bz^T)$.
By definition of our source process (\ref{gmprocess}), there is no information feedback from $\{\bz_t\}$ to $\{\bx_t\}$, and thus $I(\bx^T\rightarrow \bz^T)=I(\bx^T; \bz^T)$ holds in our setup. Hence, $I(\bx^T;\bz^T)$ can be equivalently used as an objective in (P-SRD). However, we choose to use $I(\bx^T\rightarrow \bz^T)$ for the future considerations (e.g., \cite{tanaka2015sdp}) in which $\{\bx_t\}$ is a controlled stochastic process and is dependent on $\{\bz_t\}$. In such cases, $I(\bx^T;\bz^T)$ and $I(\bx^T\rightarrow \bz^T)$ are not equal, and the latter is a more meaningful quantity in many applications.
\end{remark}

Since (P-SRD) is an infinite dimensional optimization problem, it is difficult to apply numerical methods directly. Hence, we first need to develop a structural understanding of its solution.
{\upd It turns out that the sensor-estimator separation principle still holds for (P-SRD), and this observation plays an important role in the subsequent sections. }
We are going to establish the following facts:
\begin{itemize}[leftmargin=3ex]
\item {\bf Fact 1':} A sensor-estimator separation principle holds for the Gaussian SRD problem. That is, an optimal policy $\otimes_{t=1}^T q(dz_t|x^t,z^{t-1})$ for (P-SRD) can be realized as a composition of a sensor mechanism
\begin{equation}
\label{srdchannel}
\by_t=C_t\bx_t+\bv_t, \;\; t=1, 2, \cdots, T
\end{equation}
where {\upd $\bv_t\sim\mathcal{N}(0,V_t), V_t\succ 0$ are mutually independent Gaussian random variables}, and the least mean square error estimator (Kalman filter)
\begin{equation}
\label{srdkf}
\bz_t=\mathbb{E}(\bx_t|\by^t), \;\; t=1, 2, \cdots, T.
\end{equation}
\item {\bf Fact 2':} The original optimization problem (P-SRD) over an infinite-dimensional space $\Gamma$ is reduced to an optimization problem over a finite-dimensional space of \emph{matrix-valued} signal-to-noise ratios of the sensor (\ref{srdchannel}), defined by
\begin{equation}
\label{matrixsnr}
\mathsf{SNR}_t\triangleq C_t^\top V_t^{-1} C_t \succeq 0, \;\; t=1, 2, \cdots, T.
\end{equation}
Moreover, the optimal $\{\mathsf{SNR}_t\}_{t=1}^T$, which depends on $D_t>0, t=1,\cdots,T$, can be obtained by SDP.
\end{itemize}
{\upd Unlike (\ref{eqscalarsnr}), an analytical expression of the optimal $\{\mathsf{SNR}_t\}_{t=1}^T$ may not be available.
Nevertheless, we will show that they can be easily obtained by SDP.}

\begin{figure}[t]
    \centering
    \ifdefined\FIGDOUBLECOL \includegraphics[width=\columnwidth]{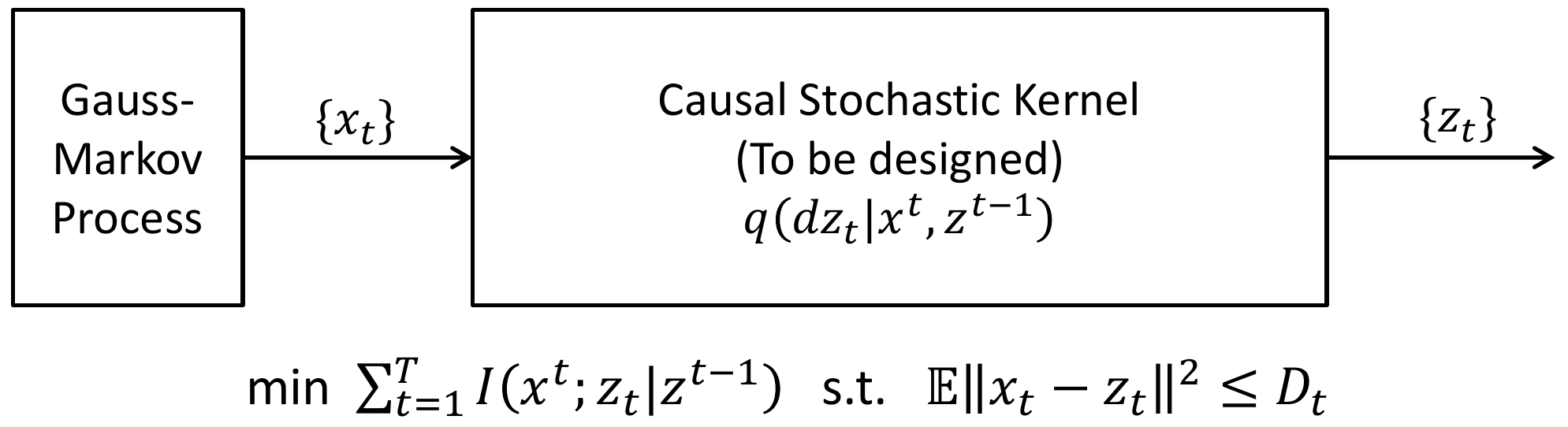} \fi
    \ifdefined\FIGSINGLECOL \includegraphics[width=.6\columnwidth]{SRD.pdf} \fi
    \caption{The Gaussian sequential rate-distortion problem (P-SRD).}
    \label{fig:SRD}
\vspace{2ex}
\end{figure}
\begin{figure}[t]
    \centering
    \ifdefined\FIGDOUBLECOL \includegraphics[width=\columnwidth]{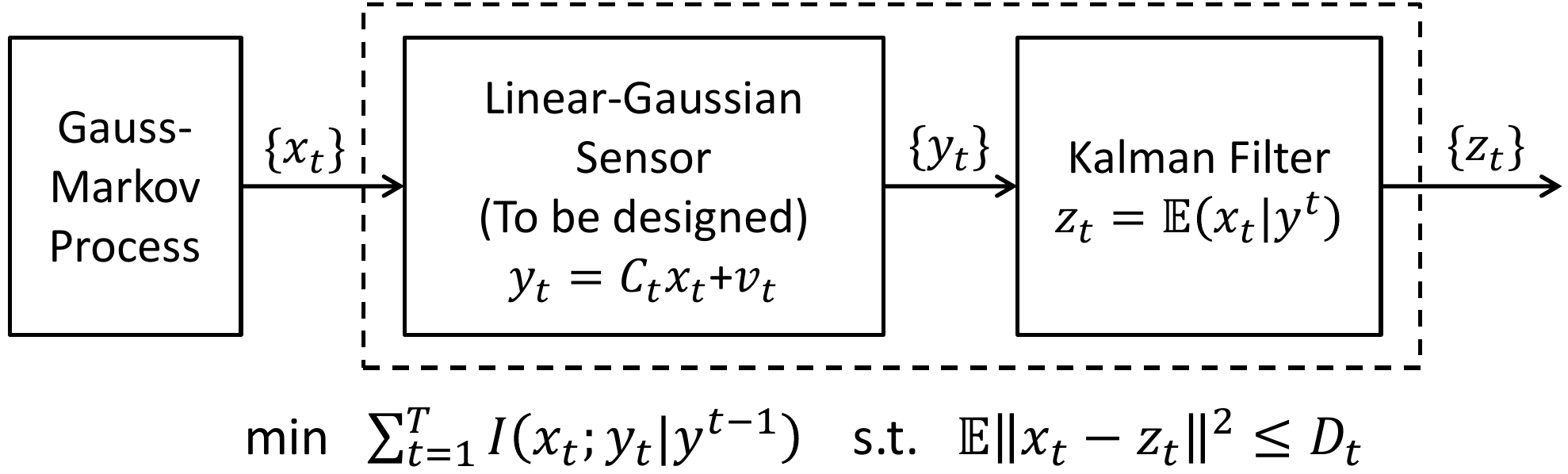} \fi
    \ifdefined\FIGSINGLECOL \includegraphics[width=.6\columnwidth]{LGS.pdf} \fi
    \caption{The linear-Gaussian sensor design problem (P-LGS).}
    \label{fig:LGS}
\vspace{2ex}
\end{figure}

\subsection{Linear-Gaussian sensor design problem}
{\upd In Section~\ref{secequivalence}, we establish the sensor-estimator separation principle. To this end, we show that (P-SRD) is equivalent to what we call the \emph{linear-Gaussian sensor design problem} (P-LGS) visualized in Fig.~\ref{fig:LGS}. 
Formally, (P-LGS) is formulated as}
\begin{subequations}
\begin{align}
\text{{\bf (P-LGS): }}\hspace{1ex} 
\min_{\gamma\in\Gamma_{\text{LGS}}} & \;\; \sum_{t=1}^T I(\bx_t;\by_t|\by^{t-1}) \label{plgs1}\\
\text{s.t. } & \;\; \mathbb{E}\|\bx_t-\bz_t\|_{\Theta_t}^2 \leq D_t \label{plgs2} 
\end{align}
\end{subequations}
where (\ref{plgs2}) is imposed for every $t=1,\cdots,T$.
We assume that $\{\by_t\}$ is produced by a linear-Gaussian sensor (\ref{srdchannel}), and $\{\bz_t\}$ is produced by  the Kalman filter (\ref{srdkf}). 
In other words, the optimization domain $\Gamma_{\text{LGS}}\subset \Gamma$ is the space of causal stochastic kernels with a separation structure (\ref{srdchannel}) and (\ref{srdkf}), which is parameterized by a sequence of matrices $\{C_t, V_t\}_{t=1}^T$.
Intuitively, $I(\bx_t;\by_t|\by^{t-1})$ in (\ref{plgs1}) can be understood as the amount of information acquired by the sensor (\ref{srdchannel}) at time $t$.
We call this problem a ``sensor design problem" because our focus is on choosing an optimal sensing gain $C_t$ in (\ref{srdchannel}) and the noise covariance $V_t$.
Notice that perfect observation with $C_t=I$ and $V_t=0$ is trivially the best to minimize the estimation error in (\ref{plgs2}) (in fact, $\mathbb{E}\|\bx_t-\bz_t\|_{\Theta_t}^2=0$ is achieved), but it incurs significant information cost (i.e., $I(\bx_t;\by_t|\by^{t-1})=+\infty$), and hence it is not an optimal solution to (P-LGS).
\begin{remark}
\label{remark1}
In (P-LGS), we search for the optimal $C_t\in \mathbb{R}^{r_t \times n}$ and $V_t \in \mathbb{S}_{++}^{r_t}$. However, the sensor dimension $r_t$ is not given \emph{a priori}, and choosing it optimally is part of the problem. In particular, if making no observation is the optimal sensing at some specific time instance $t$, we should be able to recover $r_t=0$ as an optimal solution.
\end{remark}

{\upd Although the objective functions (\ref{psrd1}) and (\ref{plgs1}) appear differently, it will be shown in Section~\ref{secequivalence} that they coincide in the domain $\Gamma_{\text{LGS}}$. Moreover, in the same section it will be shown that an optimal solution to (P-SRD) can always be found in the domain $\Gamma_{\text{LGS}}$. These observations imply that one can obtain an optimal solution to (P-SRD) by solving  (P-LGS). }

\subsection{Stationary cases}
We will also consider a time-invariant system
\begin{equation}
\label{gmprocess_stat}
\bx_{t+1}=A \bx_t+\bw_t,  t=0, 1, 2, \cdots 
\end{equation}
where $\bx_t$ is an $\mathbb{R}^{n}$-valued random variable with $\bx_0 \sim \mathcal{N}(0, P_0)$,
and $\bw_t\sim \mathcal{N}(0, W)$ is a stationary white Gaussian noise.
We assume $P_0\succ 0$ and $W\succ 0$.
Stationary and infinite horizon version of the Gaussian SRD problem is formulated as
\begin{subequations}
\label{statsrd}
\begin{align}
\min & \;\; \limsup_{T\rightarrow \infty} \frac{1}{T} I(\bx^T\rightarrow \bz^T) \\
\text{s.t.}  & \;\; \limsup_{T\rightarrow \infty} \frac{1}{T}\sum_{t=1}^T \mathbb{E}\|\bx_t-\bz_t\|_{\Theta}^2 \leq D.
\end{align}
\end{subequations}
This is an optimization over the sequence of stochastic kernels $\otimes_{t\in\mathbb{N}} \;q(dz_t|x^t, z^{t-1})$.
The optimal value of (\ref{statsrd}) as a function of the average distortion $D$ is referred to as the \emph{sequential rate-distortion function}, and is denoted by $R_{\text{SRD}}(D)$. 

Similarly, a stationary and infinite horizon version of the linear-Gaussian sensor design problem is formulated as
\begin{subequations}
\label{p2hardconststat}
\begin{align}
\min & \;\;  \limsup_{T\rightarrow \infty}\frac{1}{T} \sum_{t=1}^T  I(\bx_t;\by_t|\by^{t-1}) \\
\text{s.t.} & \;\;  \limsup_{t\rightarrow \infty}\frac{1}{T}\sum_{t=1}^T \mathbb{E}\|\bx_t-\bz_t\|_{\Theta}^2 \leq D.
\end{align}
\end{subequations} 
Here, we assume $\by_t=C_t\bx_t+\bv_t$ where $\bv_t \sim \mathcal{N}(0, V_t), V_t \succ 0$ is a mutually independent Gaussian stochastic process and $\bz_t = \mathbb{E}(\bx_t|\by^t)$.
Design variables in (\ref{p2hardconststat}) are $\{C_t, V_t\}_{t\in\mathbb{N}}$. Again, determining their dimensions is part of the problem.

\subsection{Soft- vs. hard-constrained problems}
Introducing Lagrange multipliers $\alpha_t > 0$, one can also consider a soft-constrained version of (P-SRD):
\begin{equation}
\label{srdsoft}
\min \;\; I(\bx^T\rightarrow \bz^T)+\frac{\alpha_t}{2}\mathbb{E}\|\bx_t-\bz_t\|_{\Theta_t}^2
\end{equation}
Similarly to the Lagrange multiplier theorem (e.g., Proposition 3.1.1 in \cite{bertsekas1995}), it is possible to show that there exists a set of multipliers such that an optimal solution to  (\ref{srdsoft}) is also an optimal solution to (P-SRD). We will prove this fact in Section \ref{secsdp} after we establish that both (P-SRD) and  (\ref{srdsoft}) can be transformed as finite dimensional convex optimization problems.
For this reason, we refer to both (P-SRD) and (\ref{srdsoft}) as Gaussian SRD problems.

\section{Sensor-estimator separation principle}
\label{secequivalence}
Let $f^*_{\text{SRD}}$ and $f^*_{\text{LGS}}$ be the optimal values of (P-SRD) and (P-LGS) respectively. 
In this section, we show that $f^*_{\text{SRD}}=f^*_{\text{LGS}}$,  and an optimal solution $\gamma\in\Gamma_{\text{LGS}}$ to (P-LGS) is also an optimal solution to (P-SRD).
This result  establishes the sensor-estimator separation principle (Fact 1').
We introduce another optimization problem (P-1), which serves as an intermediate step to establish this fact.
%\ifdefined\FIGDOUBLECOL
\begin{align*}
\text{\bf (P-1):}\hspace{2ex}  & \min_{\gamma\in\Gamma_1} \;\; \sum_{t=1}^T  I (\bx_t; \bz_t|\bz^{t-1}) \\
& \text{ s.t. } \;\; \mathbb{E}\|\bx_t-\bz_t\|_{\Theta_t}^2\leq D_t.
\end{align*}
The optimization is over the space $\Gamma_1$ of linear-Gaussian stochastic kernels $\gamma=\otimes_{t=1}^T \; q(dz_t|x_t, z^{t-1})$, where each stochastic kernel $q(dz_t|x_t, z^{t-1})$ is of the form
\begin{equation}
\label{pi3lin}
\bz_t=E_t \bx_t + F_{t,t-1} \bz_{t-1} + \cdots + F_{t,1} \bz_1 + \bg_t
\end{equation}
where $E_t, F_{t,t-1}, \cdots, F_{t,1}$ are some matrices with appropriate dimensions, and $\bg_t$ is a zero-mean, possibly degenerate Gaussian random variable that is independent of $\bx_0, \bw^t, \bg^{t-1}$. 
Notice that $\Gamma_{\text{LGS}}\subset \Gamma_1 \subset \Gamma$.
The underlying Gauss-Markov process $\{\bx_t\}$ is defined by (\ref{gmprocess}). Let $f_1^*$ be the optimal value of (P-1). The next lemma claims the equivalence between (P-SRD) and (P-1).
\begin{lemma}
\label{lemstep1}
\begin{itemize}[leftmargin=3ex]
\item[(i)]  If there exists $\gamma\in\Gamma$ attaining a value $f_{\text{SRD}}<+\infty$ of the objective function in (P-SRD), then there  exists $\gamma_1\in\Gamma_1$ attaining a value $f_1 \leq f_{\text{SRD}}$ of the objective function in (P-1).
\item[(ii)]  Every $\gamma_1\in\Gamma_1(\subset \Gamma)$ attaining $f_1<+\infty$ in (P-1) also attains $f_{\text{SRD}}=f_1$ in (P-SRD).
\end{itemize}
\end{lemma}

Lemma~\ref{lemstep1} is the most significant result in this section, which essentially guarantees the linearity of an optimal solution to the Gaussian SRD problems.
{\upd The proof of Lemma~\ref{lemstep1}  can be found in Appendix \ref{secprooflemstep1}. 
The basic idea of proof relies on the well-known fact that Gaussian distribution maximizes entropy when covariance is fixed. This proposition appears as Lemma 4.3 in \cite{tatikonda2004}, but we modified the proof using the Radon-Nikodym derivatives so that the proof does not require the existence of probability density functions.}
The next lemma establishes the equivalence between (P-1) and  (P-LGS).

\begin{lemma}
\label{lemstep2}
\begin{itemize}[leftmargin=3ex]
\item[(i)]  If there exists $\gamma_1\in\Gamma_1$ attaining a value $f_1<+\infty$ of the objective function in (P-1), then there exists $\gamma_{\text{LGS}}\in\Gamma_{\text{LGS}}$ attaining a value $f_{\text{LGS}}\leq f_1$ of the objective function in \mbox{(P-LGS)}.
\item[(ii)]  Every $\gamma_{\text{LGS}}\in\Gamma_{\text{LGS}}(\subset \Gamma_1)$ attaining $f_{\text{LGS}}<+\infty$ in \mbox{(P-LGS)} also attains $f_1\leq f_{\text{LGS}}$ in (P-1).
\end{itemize}
\end{lemma}

Proof of Lemma \ref{lemstep2} is in Appendix \ref{secprooflemstep2}. Combining the above two lemmas, we obtain the following consequence, which is the main proposition in this section.
It guarantees that we can alternatively solve (P-LGS) in order to solve (P-SRD).
\begin{proposition}
\label{propequivalence}
Suppose $f_{\text{SRD}}^*<+\infty$. Then there exists an optimal solution $\gamma_{\text{LGS}}\in\Gamma_{\text{LGS}} (\subset \Gamma)$ to (P-LGS).
Moreover, an optimal solution to (P-LGS) is also an optimal solution to (P-SRD), and $f_{\text{SRD}}^*=f_{\text{LGS}}^*$.
\end{proposition}

\section{SDP-based synthesis}
\label{secsdp}

{\upd In this section, we develop an efficient numerical algorithm to solve (P-LGS). Due to the preceding discussion, this is equivalent to developing an algorithm to solve (P-SRD).
Let (\ref{gmprocess}) be given. Assume temporarily that (\ref{srdchannel}) is also fixed. The Kalman filtering formula for  computing $\bz_t=\mathbb{E}(\bx_t|\by^t)$ is}
\begin{align*}
&\bz_t= \bz_{t|t-1}\!+\!P_{t|t-1}C_t^\top (C_tP_{t|t-1}C_t^\top \!+\!V_t)^{-1}(\by_t\!-\!C_t\bz_{t|t-1}) \\
&\bz_{t|t-1}=A_{t-1}\bz_{t-1}
\end{align*}
where $P_{t|t-1}$ is the covariance matrix of $\bx_t-\mathbb{E}(\bx_t|\by^{t-1})$, which can be recursively computed as
\begin{subequations}
\label{defmatp}
\begin{align}
&P_{t|t-1}=A_{t-1}P_{t-1|t-1}A_{t-1}^\top+W_{t-1} \\
&P_{t|t}=(P_{t|t-1}^{-1}+\mathsf{SNR}_t)^{-1}
\end{align}
\end{subequations}
for $t=1,\cdots, T$ with $P_{0|0}=P_0$. The variable $\mathsf{SNR}_t$ is defined by (\ref{matrixsnr}).
Using these quantities, mutual information terms in (\ref{plgs1}) can be explicitly written as
\ifdefined\FIGDOUBLECOL 
\begin{align*}
&I(\bx_t;\by_t|\by^{t-1})  \\
&=h(\bx_t|\by^{t-1})-h(\bx_t|\by^t)  \\
&= \frac{1}{2}\!\log\det (A_{t-1}P_{t-1|t-1}A_{t-1}^\top \!+\! W_{t-1})\!-\!\frac{1}{2}\! \log\det P_{t|t}.
\end{align*}
\fi
\ifdefined\FIGSINGLECOL  
\begin{align*}
I(\bx_t;\by_t|\by^{t-1})  &=h(\bx_t|\by^{t-1})-h(\bx_t|\by^t)  \\
&= \frac{1}{2}\!\log\det (A_{t\!-\!1}P_{t\!-\!1|t\!-\!1}A_{t\!-\!1}^\top \!+\! W_{t\!-\!1})\!-\!\frac{1}{2}\! \log\det P_{t|t}.
\end{align*}
\fi
Note that $W_t\succ 0$ and $V_t\succ 0$ guarantee that both differential entropy terms are finite.
Hence, (P-LGS) is equivalent to the following optimization problem in terms of the variables $\{\textsf{SNR}_t, P_{t|t} \}_{t=1}^T$:
\begin{subequations}
\label{optprob1}
\begin{align}
\min  &\; \sum_{t=1}^T \frac{1}{2}\log\det (A_{t-1}P_{t-1|t-1}A_{t-1}^\top +W_{t-1})\! \nonumber \\[-2ex]
&\hspace{26ex} -\frac{1}{2} \log\det P_{t|t} \\ \vspace{-3ex}
\text{s.t.} & \;\; P_{t|t}^{-1}\!=\!(A_{t-1}P_{t-1|t-1}A_{t-1}^\top \!+\! W_{t-1})^{-1}\!+\!\mathsf{SNR}_t \label{optprob1cons} \\
& \;\;  \mathsf{SNR}_t \succeq 0, \text{Tr}(\Theta_tP_{t|t})\leq D_t.
\end{align}
\end{subequations}
Equality (\ref{optprob1cons}) is obtained by eliminating $P_{t|t-1}$ from (\ref{defmatp}).
At this point, one may note that (\ref{optprob1}) can be viewed as an optimal control problem with state $P_{t|t}$ and control input $\mathsf{SNR}_t$. Naturally, dynamic programming approach has been proposed in the literature in similar contexts \cite{bucy1980distortion,TatikondaThesis,mahajan2009optimal,Gorantla2011information}. 
Alternatively, we next propose a method to transform (\ref{optprob1}) into an SDP problem.
This allows us to solve (P-SRD) using standard SDP solvers, which is now a mature technology.

\subsection{SRD optimization as max-det problem}
Now we show that (\ref{optprob1}) can be converted to a determinant maximization problem \cite{vandenberghe1998} subject to linear matrix inequality constraints. 
The first step is to transform (\ref{optprob1}) into an optimization problem in terms of $\{P_{t|t}\}_{t=1}^T$ only. This is possible by simply replacing the nonlinear equality constraint (\ref{optprob1cons}) 
with a linear inequality constraint
\[ 0 \prec P_{t|t} \preceq  A_{t-1}P_{t-1|t-1}A_{t-1}^\top + W_{t-1}. \]
This replacement eliminates $\mathsf{SNR}_t$ from (\ref{optprob1}) giving us:
\begin{subequations}
\label{optprob2}
\begin{align}
\min  &\; \sum_{t=1}^T \frac{1}{2}\log\det (A_{t-1}P_{t-1|t-1}A_{t-1}^\top +W_{t-1})\! \nonumber \\[-2ex]
&\hspace{26ex} -\frac{1}{2} \log\det P_{t|t}  \label{optprob2obj} \\
 \text{s.t.} & \;\; 0 \prec P_{t|t} \preceq  A_{t-1}P_{t-1|t-1}A_{t-1}^\top + W_{t-1} \label{optprob2cons}  \\
& \;\; \text{Tr}(\Theta_tP_{t|t})\leq D_t.
\end{align}
\end{subequations}
Note that (\ref{optprob1}) and (\ref{optprob2}) are mathematically equivalent, since eliminated SNR variables can be easily constructed from $\{P_{t|t}\}_{t=1}^T$ through
\begin{equation}
\label{cvconst}
 \mathsf{SNR}_t= P_{t|t}^{-1}-(A_{t-1}P_{t-1|t-1}A_{t-1}^\top + W_{t-1})^{-1}.
\end{equation}
The second step is to rewrite the objective function  (\ref{optprob2obj}).
Regrouping terms,  (\ref{optprob2obj}) can be written as a summation of the initial cost $\frac{1}{2}\log \det (A_0 P_{0|0}A_0^\top + W_0)$, the final cost $-\frac{1}{2}\log\det P_{T|T}$, and stage-wise costs
\begin{equation}
\label{beforedetlemma}
\frac{1}{2}\log\det (A_t P_{t|t} A_t^\top + W_t)-\frac{1}{2}\log\det P_{t|t}
\end{equation}
for $t=1,\cdots, T-1$.
Applying the matrix determinant lemma (e.g., Theorem 18.1.1 in \cite{harville1997matrix}), (\ref{beforedetlemma}) can be rewritten as
\begin{equation}
\label{afterdetlemma}
\frac{1}{2}\log\det W_t - \frac{1}{2}\log\det (P_{t|t}^{-1}+A_t^\top W_t^{-1}A_t)^{-1}. 
\end{equation}
Due to the monotonicity of the determinant function,  (\ref{afterdetlemma}) is equal to the optimal value of
\begin{subequations}
\begin{align}
\min & \;\; \frac{1}{2}\log\det W_t-\frac{1}{2}\log\det \Pi_t \\
\text{s.t. } & \;\; 0 \prec \Pi_t \preceq (P_{t|t}^{-1}+A_t^\top W_t^{-1}A_t)^{-1}. \label{piconstinv}
\end{align}
\end{subequations}
Applying the matrix inversion lemma, (\ref{piconstinv}) is equivalent to
$0 \prec \Pi_t \preceq   P_{t|t}-P_{t|t}A_t^\top (W_t+A_tP_{t|t}A_t^\top )^{-1}A_tP_{t|t}$, 
which is further equivalent to 
\begin{equation}
\label{lmi1}
\left[ \begin{array}{cc}P_{t|t}-\Pi_t & P_{t|t}A_t^\top \\
A_tP_{t|t} & W_t+A_t P_{t|t}A_t^\top  \end{array}\right] \succeq 0, \;\;\; \Pi_t \succ 0.
\end{equation}
Note that (\ref{lmi1}) is a linear matrix inequality (LMI) condition.
The above discussion leads to the following conclusion.

\begin{theorem}
\label{propsdp}
An optimal solution to (P-LGS) can be constructed by solving the following determinant maximization problem with decision variables $\{P_{t|t}, \Pi_t\}_{t=1}^T$:
\begin{subequations}
\label{optprob3hard}
\begin{align}
\min & \;\; -\sum_{t=1}^T \frac{1}{2}\log\det \Pi_t +c \\
\text{s.t.} & \;\; \Pi_t \succ  0, \;\; t=1, ... , T \label{optprob3hard2} \\
& \;\; P_{t+1|t+1}\preceq A_t P_{t|t}A_t^\top \!+\!W_t, \;\; t=0, ... , T-1 \\
& \; \left[\! \!\!\begin{array}{cc}P_{t|t}\!-\!\Pi_t \!\!\!&\!\!\! P_{t|t}A_t^\top \\
A_tP_{t|t} \!\!\!&\!\!\! W_t\!+\!A_t P_{t|t}A_t^\top  \end{array}\!\!\!\right] \! \succeq\! 0, \;\; t\!=\!1, ... , T\!-\!1 \label{averageconst}\\
& \;\;  \text{Tr}(\Theta_t P_{t|t}) \leq D_t, \;\; t\!=\!1, ... , T \\
& \;\; P_{T|T}=\Pi_T, \label{optprob3hard3}
\end{align}
\end{subequations} 
where 
$ c=\frac{1}{2}\log\det (A_0P_{0|0}A_0^\top +W_0) + \sum_{t=1}^{T-1} \frac{1}{2}\log\det W_t$
is a constant. 
The optimal sequence $\{\mathsf{SNR}_t\}_{t=1}^T$ can be reconstructed from (\ref{cvconst}), from which $\{C_t, V_t\}_{t=1}^T$ satisfying (\ref{matrixsnr}) can be reconstructed via the singular value decomposition.
An optimal solution to (P-LGS) is obtained as a composition of (\ref{srdchannel}) and (\ref{srdkf}).
\end{theorem}
\begin{remark}
Under the assumption that $W_t\succ 0, D_t>0$ for every $t=1,\cdots,T$, the max-det problem (\ref{optprob3hard}) is always strictly feasible and there exists an optimal solution.\footnote{To see the strict feasibility, consider $P_{t|t}=\delta I$ for $t=1,\cdots,T-1$ and $\Pi_t=\delta^2 I$ for $t=1,\cdots,T$ with sufficiently small $\delta>0$. The constraint set defined by (\ref{optprob3hard2})-(\ref{optprob3hard3}) can be made compact by replacing (\ref{optprob3hard2}) with $\Pi_t\succeq \epsilon I$ without altering the result. Thus the existence of an optimal solution is guaranteed by the Weierstrass theorem.}
 Invoking Proposition~\ref{propequivalence}, we have thus shown by construction that there always exists an optimal solution to (P-SRD) under this assumption.
\end{remark}
\begin{remark}
As we mentioned in Remark \ref{remark1}, choosing an appropriate  dimension $r_t$ of the sensor output (\ref{srdchannel}) is part of (P-LGS). It can be easily seen from Theorem \ref{propsdp} that the minimum sensor dimension to achieve the optimality in (P-LGS) is given by $r_t=\text{rank}(\textsf{SNR}_t)$.
\end{remark}

Using the same technique,  the soft-constrained version of the problem (\ref{srdsoft}) can be formulated as:
\begin{subequations}
\label{optprob3}
\begin{align}
\min & \;\; \sum_{t=1}^T \left(\frac{\alpha_t}{2}\text{Tr}(\Theta_t P_{t|t})-\frac{1}{2}\log\det \Pi_t \right)+c \\
\text{s.t.} & \;\; \Pi_t \succ  0, \;\; t=1, ... , T \label{optprob32} \\
& \;\; P_{t+1|t+1}\preceq A_t P_{t|t}A_t^\top +W_t, \;\; t=0, ...  , T-1 \\
& \; \left[\! \!\!\begin{array}{cc}P_{t|t}\!-\!\Pi_t \!\!\!&\!\!\! P_{t|t}A_t^\top \\
A_tP_{t|t} \!\!\!&\!\!\! W_t\!+\!A_t P_{t|t}A_t^\top  \end{array}\!\!\!\right] \! \succeq\! 0, \;\; t\!=\!1, ... , T\!-\!1  \\
& \;\; P_{T|T}=\Pi_T
\end{align}
\end{subequations} 

The next proposition claims that (\ref{optprob3hard})  and  (\ref{optprob3})  admit the same optimal solution provided Lagrange multipliers $\alpha_t, t=1,\cdots, T$,  are chosen correctly. This further implies that, with the same choice of $\alpha_t$, two versions of the Gaussian SRD problems  (P-SRD) and (\ref{srdsoft}) are equivalent.
\begin{proposition}
Suppose $W_t\succ 0, D_t > 0$ for $t=1,\cdots, T$. Then, there exist $\alpha_t, t=1,\cdots, T$ such that an optimal solution to (\ref{optprob3hard}) is also an optimal solution to (\ref{optprob3}). 
\end{proposition}
\ifdefined\SHORTVERSION
\begin{proof}
Both (\ref{optprob3}) and (\ref{optprob3hard}) are strictly feasible. The result follows using the fact that the Slater's constraint qualification is satisfied for this problem, which guarantees that strong duality holds and the dual optimum is attained \cite{boyd2009}. 
%Details are available in \cite{Supplementalmaterials}.
\end{proof}
\fi
\ifdefined\LONGVERSION
\begin{proof}
First, notice that both (\ref{optprob3hard}) and (\ref{optprob3}) are strictly feasible with finite optimal values. Also, conditions (\ref{optprob32}) and  (\ref{optprob3hard2}) can be replaced with $\Pi_t \succeq \epsilon I$ with sufficiently small $\epsilon>0$ without altering problems. After this replacement, the feasible domains of  (\ref{optprob3}) and (\ref{optprob3hard}) become compact, since they are clearly bounded. Hence both  (\ref{optprob3}) and (\ref{optprob3hard}) have optimal solutions.
After (\ref{optprob3hard2}) is replaced with $\Pi_t \succeq \epsilon I$, and letting $n$ be the number of real variables in problem (\ref{optprob3hard}),  (\ref{optprob3hard}) can be written as an optimization problem in terms of a vector $x\in \mathbb{R}^n$ as
\begin{subequations}
\label{probfghard}
\begin{align}
\min_{x\in \mathbb{R}^n} & \;\; f(x) \\
\text{ s.t. } & \;\; g(x) \preceq 0, \;\; h_t(x) \leq D_t \;\; t=1, \cdots, T
\end{align}
\end{subequations}
where $f(\cdot)$ is convex, $g(\cdot)$ is affine, and $h_t(\cdot), t=1,\cdots, T$ are linear in $x$.
Denote by $x^*$ an optimal solution to  (\ref{probfghard}). Let $Y \succeq 0, z_t\geq 0, t=1,\cdots, T$ be Lagrangian multipliers and
\[L(x; Y, z)=f(x)+\left< Y, g(x)\right> + \sum_{t=1}^T z_t(h_t(x)-D_t)\]
be the Lagrangian corresponding to the problem (\ref{probfghard}).
Due to the strict feasibility of  (\ref{optprob3hard}),  problem (\ref{probfghard}) satisfies Slater's constraint qualification. This implies that strong duality holds and that dual optimum is attained \cite{boyd2009}. Denoting by $(Y^*, z^*)$ a dual optimal solution, this implies 
\begin{equation}
\label{eqsaddle}
L(x^* ; Y, z) \leq L(x^* ; Y^*, z^*)\leq L(x; Y^*, z^*)
\end{equation}
for all $x\in \mathbb{R}^n$, $Y \succeq 0$, and $z_t \geq 0, t=1, \cdots, T$.
We need to prove that there exist $\alpha_t, t=1,\cdots, T$ such that $x^*$ is also an optimal solution to 
\begin{subequations}
\label{probfg}
\begin{align}
\min & \;\; f(x) + \sum_{t=1}^T \frac{\alpha_t}{2} (h_t(x) - D_t)\\
\text{ s.t. } & \;\; g(x) \preceq 0.
\end{align}
\end{subequations}
If we write $L(x;Y,z)=:L_z(x;Y)$, the corresponding Lagrangian to (\ref{probfg}) is $L_{\alpha/2}(x;Y)$. By choosing $\alpha_t$ as $\alpha_t/2=z_t^*, t=1, \cdots, T$, it follows from (\ref{eqsaddle}) that 
\[
L_{z^*}(x^* ; Y,) \leq L_{z^*}(x^* ; Y^*)\leq L_{z^*}(x; Y^*).
\]
Thus, $x^*$ is an optimal solution to (\ref{probfg}) when $\alpha_t/2=z_t^*, t=1, \cdots, T$. 
\end{proof}
\fi

\subsection{Max-det problem as SDP}
Strictly speaking, optimization problems (\ref{optprob3hard}) and (\ref{optprob3}) are in the class of determinant maximization problems \cite{vandenberghe1998}, but not in the standard form of the SDP.\footnote{In the standard form, SDP is an optimization problem of the form $\min \; \left<C, X\right>  \text{ s.t. } \mathcal{A}(X)=B, X \succeq 0$.}  However, they can be considered as SDPs in a broader sense for the following reasons. First, the hard constrained version (\ref{optprob3hard}) can be indeed transformed into a standard SDP problem.
This conversion is possible by following the discussion in Chapter 4 of \cite{ben2001}.
Second, sophisticated and efficient algorithms based on the interior-point method for SDP can almost directly  be applied to max-det problems as well. 
In fact, off-the-shelf SDP solvers such as SDPT3 \cite{tutuncu2003solving} have built-in functions to handle log-determinant terms directly.

Recall that (P-LGS) and (P-SRD) have a common optimal solution.
Hence, Proposition \ref{propsdp} shows that both (P-LGS) and (P-SRD) are essentially solvable via SDP, which is much stronger than merely saying that they are convex problems. Note that convexity alone does not guarantee the existence of an efficient optimization algorithm.
\begin{figure}[t]
    \centering
    \ifdefined\FIGDOUBLECOL \includegraphics[width=.8\columnwidth]{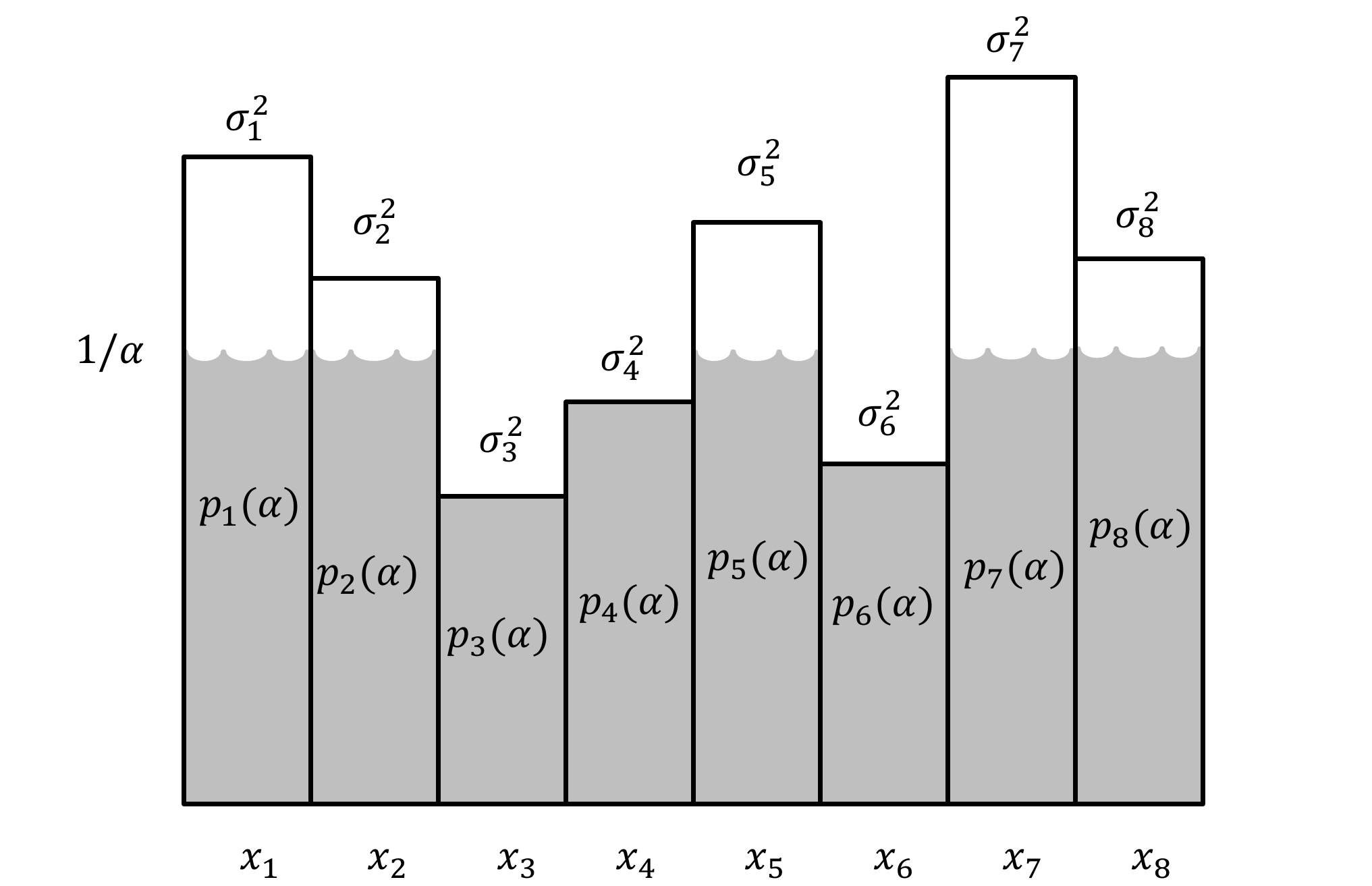} \fi
    \ifdefined\FIGSINGLECOL \includegraphics[width=.5\columnwidth]{waterfilling.pdf} \fi
    \caption{Reverse water-filling solution to the Gaussian rate-distortion problem.}
    \label{fig:waterfilling}
\end{figure}

\subsection{Complexity analysis}

In this section, we briefly consider the arithmetic complexity  (i.e., the worst case number of arithmetic operations needed to obtain an $\epsilon$-optimal solution) of problem (\ref{optprob3hard}), and how it grows as the horizon length $T$ grows when the dimensions of the Gauss-Markov process (\ref{gmprocess}) are fixed to $n_t=n \; \forall t=1,\cdots,T$.
For a preliminary analysis, it would be natural for us to resort to the existing interior-point method literature (e.g., \cite{ben2001,renegar2001}). Interior-point methods for the determinant maximization problem are already considered in \cite{vandenberghe1998,toh1999,tsuchiya2007}. 
The most computationally expensive step in the interior-point method is the Cholesky factorization involved in the Newton steps, which requires $\mathcal{O}(T^3)$ operations in general.
However, it is possible to exploit the sparsity of coefficient matrices in the SDPs to reduce operation counts \cite{fukuda2001exploiting, nakata2003exploiting,vandenberghe2015chordal}.
By exploiting the structure of our SDP formulation (\ref{optprob3hard}), it is theoretically expected that there exists a specialized interior-point method algorithm for (\ref{optprob3hard}) whose arithmetic complexity is $\mathcal{O}(T \log(1/\epsilon))$. However, more careful study and computational experiments are needed to verify this conjecture.

\subsection{Single stage problem}
When $T=1$, the result of Proposition \ref{propsdp} recovers the well-known ``reverse water-filling'' solution for the standard Gaussian rate-distortion problem \cite{CoverThomas}. 
To see this, notice that $T=1$ reduces problem (\ref{optprob3}) to
\begin{align*}
\min & \;\; \text{Tr} P - \frac{1}{\alpha} \log \det P \\
\text{s.t.} & \;\; 0 \preceq P \preceq \diag (\sigma_1^2, \cdots, \sigma_n^2).
\end{align*}
Here, we have already assumed $\Theta=I$ and $AP_0A^\top+W=\diag (\sigma_1^2, \cdots, \sigma_n^2) \succeq 0$. This does not result in loss of generality, since otherwise a change of variables $P\leftarrow U\Theta^{\frac{1}{2}}P\Theta^{\frac{1}{2}}U^\top$, where $U$ is an orthonormal matrix that makes $U\Theta^{\frac{1}{2}}(AP_0A^\top+W)\Theta^{\frac{1}{2}}U^\top$ diagonal, converts the problem into the above form. 
For any positive definite matrix $P$, Hadamard's inequality (e.g., \cite{CoverThomas}) states that 
$
\det P \leq \prod_i P_{ii}
$
and the equality holds if and only if the matrix is diagonal. 
Hence, if diagonal elements of $P$ are fixed, $\det P$ is maximized by setting all off-diagonal entries zero.
Thus, the optimal solution to the above problem is diagonal.
Writing $P=\diag(p_1,\cdots, p_n)$, the problem is decomposed as $n$ independent optimization problems, each of which minimizes $p_i-\frac{1}{\alpha} \log p_i$ subject to $0\leq p_i \leq \sigma_i^2$. It is easy to see that the optimal solution is $p_i=\min(1/\alpha, \sigma_i^2)$. This is the closed-form solution to (P-LGS) with $T=1$, and its pictorial interpretation is shown in Fig. \ref{fig:waterfilling}.
This solution also indicates the optimal sensing formula is given by $\by=C\bx+\bv, \bv\sim \mathcal{N}(0,V)$, where $C$ and $V$ satisfy
\ifdefined\FIGDOUBLECOL
\begin{align*}
C^\top V^{-1} C &= P^{-1} - (AP_0A^\top+W)^{-1} \\
&=\diag_{1\leq i \leq n} \left( \max \left\{0, \alpha-\frac{1}{\sigma_i^2} \right\}\right).
\end{align*}
\fi
\ifdefined\FIGSINGLECOL  
\[ C^\top V^{-1} C = P^{-1} - (AP_0A^\top+W)^{-1} =\diag_{1\leq i \leq n} \left( \max \{0, \alpha-\frac{1}{\sigma_i^2} \}\right). \]
\fi
In particular, we have $\text{dim}(\by)=\text{rank}(C^\top V^{-1} C)=\text{card}\{ i: \sigma_i^2>\frac{1}{\alpha}\}$, indicating that the optimal dimension of $y$ monotonically decreases as the ``price of information" $1/\alpha$ increases.

\section{Stationary problems}
\label{secstationary}

\begin{figure}[t]
    \centering
    \ifdefined\FIGDOUBLECOL \includegraphics[width=\columnwidth]{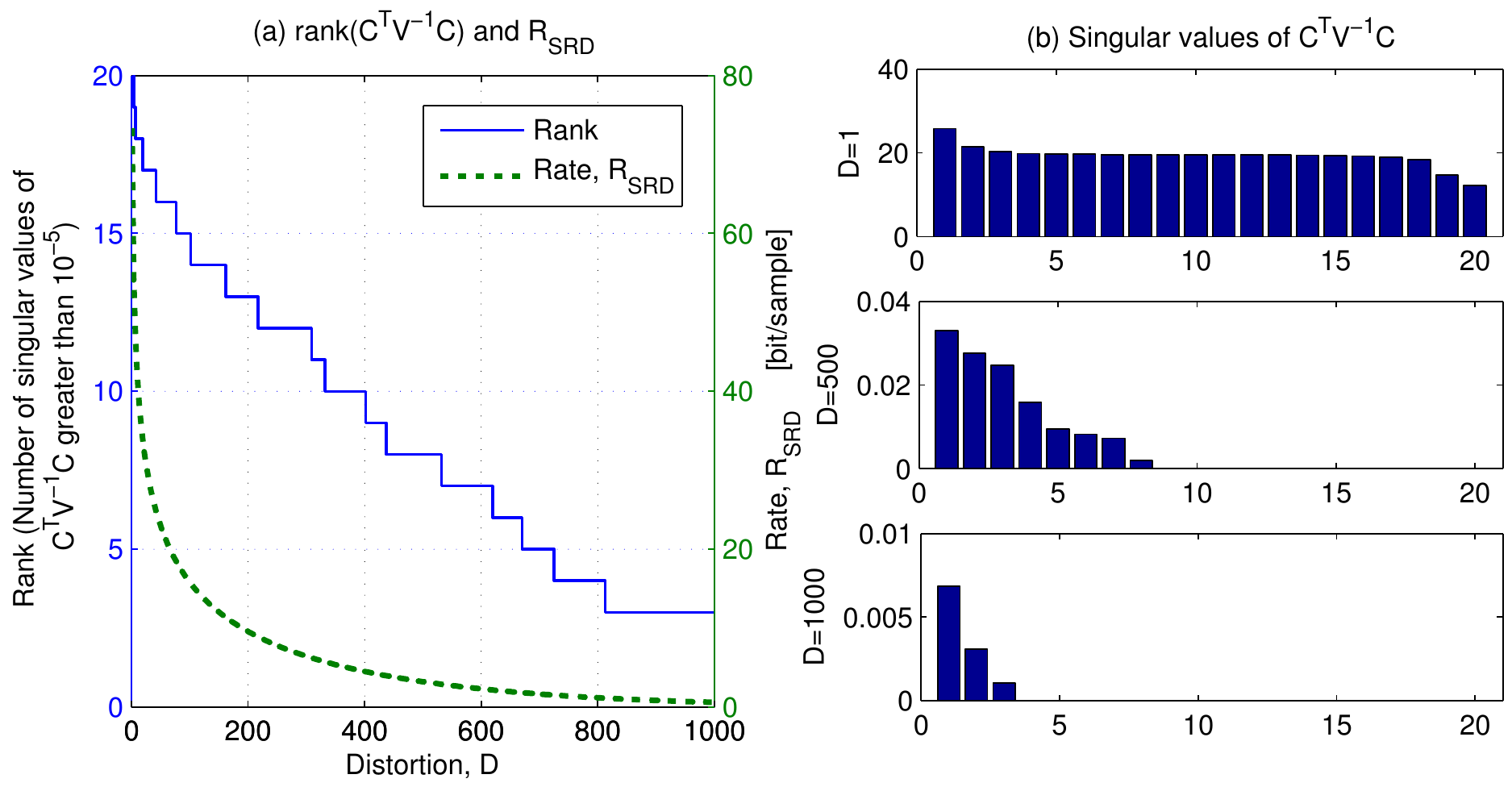} \fi
    \ifdefined\FIGSINGLECOL \includegraphics[width=.8\columnwidth]{rank.pdf} \fi
    \caption{Numerical experiments on rank monotonicity. 20 dimensional Gaussian process is randomly generated and $\textsf{SNR}=C^\top V^{-1}C$ is constructed for various $D$. Observe $\text{rank}(\textsf{SNR})$ tends to decrease as $D$ increase.}
    \label{fig:rank}
\end{figure}

\subsection{Sequential rate-distortion function}
We are often interested in infinite-horizon Gaussian SRD problems (\ref{statsrd}).
Assuming that $(A,\Theta)$ is a detectable pair, it can be shown that  (\ref{statsrd}) is equivalent to the infinite-horizon linear-Gaussian sensor design problem (\ref{p2hardconststat}) \cite{srdstationary}. Moreover, \cite{srdstationary} shows that  (\ref{statsrd}) and (\ref{p2hardconststat}) admit an optimal solution that can be realized as a composition of a \emph{time-invariant} sensor mechanism $\by_t=C\bx_t+\bv_t$ with i.i.d. process $\bv_t\sim\mathcal{N}(0,V)$ and a  \emph{time-invariant} Kalman filter.
Hence, it is enough to minimize the average cost per stage, which leads to the following  simpler problem. 
\begin{subequations}
\label{stationarylmi}
\begin{align}
R_{\text{SRD}}(D)= \min & \; - \frac{1}{2} \log \det \Pi + \frac{1}{2} \log \det W\\
\text{s.t.} & \;\; \Pi \succ 0 \\
& \;\; P \preceq APA^\top+W \\
& \;\; \text{Tr}(\Theta P) \leq D \\
& \;\; \left[ \begin{array}{cc}
P-\Pi & PA^\top \\
AP & APA^\top+W \end{array}\right]\succeq 0.
\end{align}
\end{subequations}
To confirm (\ref{stationarylmi}) is compatible with the existing result, consider a scalar system with $A=a, W=w, P=p$ and $\Theta=1$. In this case, a closed-form expression of the SRD function is known in the literature \cite{TatikondaThesis} \cite{derpich2012}, which is given by
\begin{equation}
\label{scalarsrdf}
R_{\text{SRD}}(D)=\min\left\{0, \frac{1}{2} \log(a^2+\frac{w}{D})\right\}.
\end{equation}
For a scalar system,  (\ref{stationarylmi}) further simplifies to 
\begin{subequations}
\label{statscalarsrd}
\begin{align}
\min & \;\; \log(a^2+\frac{w}{p}) \\
\text{s.t.} & \;\; 0 < p \leq a^2 p + w, \;\; p\leq D.
\end{align}
\end{subequations}
It is elementary to verify that the optimal value of (\ref{statscalarsrd}) is $\log(a^2+\frac{w}{D})$ if $1-\frac{w}{D}\leq a^2$, while it is $0$ if $0\leq a^2 < 1-\frac{w}{D}$. Hence, it can be compactly written as $\min\{0, \frac{1}{2} \log(a^2+\frac{w}{D})\}$, and the result recovers (\ref{scalarsrdf}).
Alternative representations of the SRD function (\ref{stationarylmi}) for stationary multi-dimensional Gauss-Markov processes when $\Theta=I$ are reported in \cite[Section IV-B]{tatikonda2004} and \cite[Section VI]{charalambous2014nonanticipative}.

\subsection{Rank monotonicity}
Using an optimal solution to (\ref{stationarylmi}) the optimal sensing matrices $C$ and $V$ are recovered from $C^\top V^{-1}C=P^{-1}-(APA^\top+W)^{-1}$. In particular, $\text{dim}(\by)=\text{rank}(C^\top V^{-1}C)$ determines the optimal dimension of the measurement vector. Similarly to the case of single stage problems, this rank has a tendency to decrease as $D$ increases. 
A typical numerical behavior is shown in Figure \ref{fig:rank}.
We do not attempt to prove the rank monotonicity here.

\section{Applications and related works}
\label{secapplications}
\subsection{Zero-delay source coding}
\label{secapplicationszerodelay}

{\upd 
SRD theory plays an important role in the rate analysis of zero-delay source coding schemes.
For each $t=1,2,\cdots$, let
\[
\mathcal{B}_t \subset \{0, 1, 00, 01, 10, 11, 000, \cdots \}
\]
be a set of variable-length uniquely decodable codewords.
Assume that ${\bf{b}}_t \in \mathcal{B}_t$ for $t=1,2,\cdots$, and let $l_t$ be the length of ${\bf{b}}_t$.
A zero-delay binary coder is a pair of a sequence of encoders $e_t(db_t|x^t,b^{t-1})$, i.e., stochastic kernels on $\mathcal{B}_t$ given $\mathcal{X}^t \times \mathcal{B}^{t-1}$, and a sequence of decoders $d_t(dz_t|b^t, z^{t-1})$, i.e., stochastic kernels on $\mathcal{Z}_t$ given $\mathcal{B}^t\times \mathcal{Z}^{t-1}$.
The \emph{zero-delay rate-distortion region} for the Gauss-Markov process (\ref{gmprocess_stat}) is the epigraph of the function
\begin{align*}
R_{\text{SRD}}^{\text{op}}(D)=&\inf_{\{\mathcal{B}_t, e_t, d_t\}_{t=1}^\infty} \; \limsup_{T\rightarrow \infty} \frac{1}{T}\sum_{t=1}^T \mathbb{E}(l_t) \\
& \hspace{5ex} \text{s.t.} \; \limsup_{T\rightarrow \infty} \frac{1}{T}\sum_{t=1}^T \mathbb{E}\|\bx_t-\bz_t\|_\Theta^2\leq D.
\end{align*}
The SRD function is a lower bound of the achievable rate. Indeed,  $R_{\text{SRD}}(D) \leq R_{\text{SRD}}^{\text{op}}(D) \; \forall D>0$ can be shown straightforwardly as
\begin{subequations}
\begin{align}
I(\bx^T \rightarrow \bz^T)&= I(\bx^T;\bz^T) \label{eqlength1}\\
&\leq I(\bx^T;{\bf{b}}^T) \label{eqlength2}\\
&=H({\bf{b}}^T)- H({\bf{b}}^T | \bx^T) \label{eqlength3}\\
&\leq H({\bf{b}}^T) \label{eqlength4}\\
&\leq \sum\nolimits_{t=1}^T H({\bf{b}}_t) \label{eqlength5}\\
&\leq \sum\nolimits_{t=1}^T \mathbb{E}(l_t) \label{eqlength6}
\end{align}
\end{subequations}
where (\ref{eqlength1}) holds since there is no feedback from the process $\{\bz_t\}$ to $\{\bx_t\}$ (Remark~\ref{remark_directedinfo}), (\ref{eqlength2}) follows from the data processing inequality, (\ref{eqlength4}) holds since conditional entropy is non-negative, and (\ref{eqlength5}) is due to the chain rule for entropy. 
The final inequality (\ref{eqlength6}) holds since the expected length of a uniquely decodable code is lower bounded by its entropy \cite[Theorem 5.3,1]{CoverThomas}.

In general, $R_{\text{SRD}}(D)$ and $R_{\text{SRD}}^{\text{op}}(D)$ do not coincide. Nevertheless, by constructing an appropriate entropy-coded dithered quantizer (ECDQ), it is shown in \cite{derpich2012} that $R_{\text{SRD}}^{\text{op}}(D)$ does not exceed $R_{\text{SRD}}(D)$ more than a constant due to the ``space-filling loss" of the lattice quantizer and the loss of entropy coding.
}

\subsection{Networked control theory}
{\upd
Zero-delay source/channel coding technologies are crucial in networked control systems \cite{nair2007,baillieul2007,hespanha2007survey,yuksel2013stochastic}. Gaussian SRD theory plays an important role in the LQG control problems with information theoretic constraints \cite{tatikonda2004}.
It is shown in \cite{silva2011} that an LQG control problem in which observed data must be transmitted to the controller over a noiseless binary channel is closely related to the LQG control problem with directed information constraints.
The latter problem is addressed in \cite{tanaka2015sdp} using the SDP-based algorithm presented in this paper. 
In \cite{tanaka2015sdp}, the problem is viewed as a sensor-controller joint design problem in which
directed information from the state process to the control input is minimized.\footnote{The problem considered in \cite{tanaka2015sdp} is different from the sensor-controller joint design problems considered in \cite{Bansal1989} and \cite{miller1997}.}
}

\subsection{Experimental design/Sensor scheduling}
\label{secexpdesign}
{\upd
In this subsection, we compare the linear-Gaussian sensor design problem (P-LGS) with different types of sensor design/selection problems considered in the literature. 

A problem of selecting the best subset of sensors to observe a random variable in order to minimize the estimation error and its convex relaxations are considered in \cite{vandenberghe1998}. A sensor selection problem for a linear dynamical system is considered in \cite{7151797}, where submodularity of the objective function is exploited.
Dynamic sensor scheduling problems are also considered in the literature. In \cite{vitus2012}, an efficient algorithm to explore branches of the scheduling tree is proposed. In \cite{gupta2006}, a stochastic sensor selection strategy that minimizes the expected error covariance is considered.

The linear-Gaussian sensor design problem (P-LGS) is different from these sensor selection/scheduling problems in that it is essentially a continuous optimization problem (since matrices $\{C_t, V_t\}_{t=1}^T$ can be freely chosen), and the objective is to minimize an information-theoretic cost (\ref{plgs1}).
}

\section{Numerical Simulations}
\label{secexample}
In this section, we consider two numerical examples to demonstrate how the SDP-based formulation of the Gaussian SRD problem can be used to calculate the minimal communication bandwidth required for the real-time estimation with desired accuracy.
\begin{figure}[t]
    \centering
    \ifdefined\FIGDOUBLECOL \includegraphics[width=0.8\columnwidth]{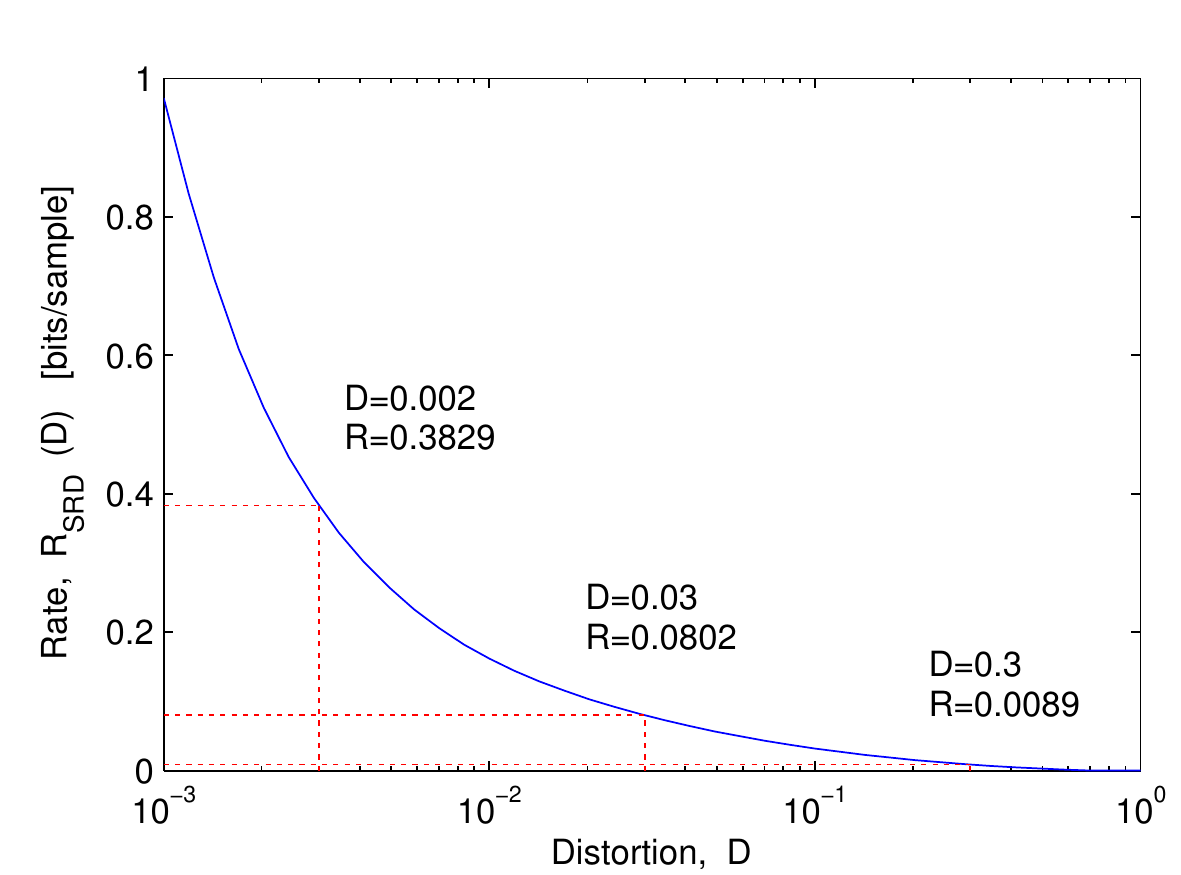} \fi
    \ifdefined\FIGSINGLECOL \includegraphics[width=0.5\columnwidth]{pend2_SRDF.pdf} \fi
    \caption{Sequential rate-distortion function for the noisy double pendulum. }
    \label{fig:pend_srdf}
\end{figure}
\begin{figure}[t]
    \centering
    \ifdefined\FIGDOUBLECOL \includegraphics[width=\columnwidth]{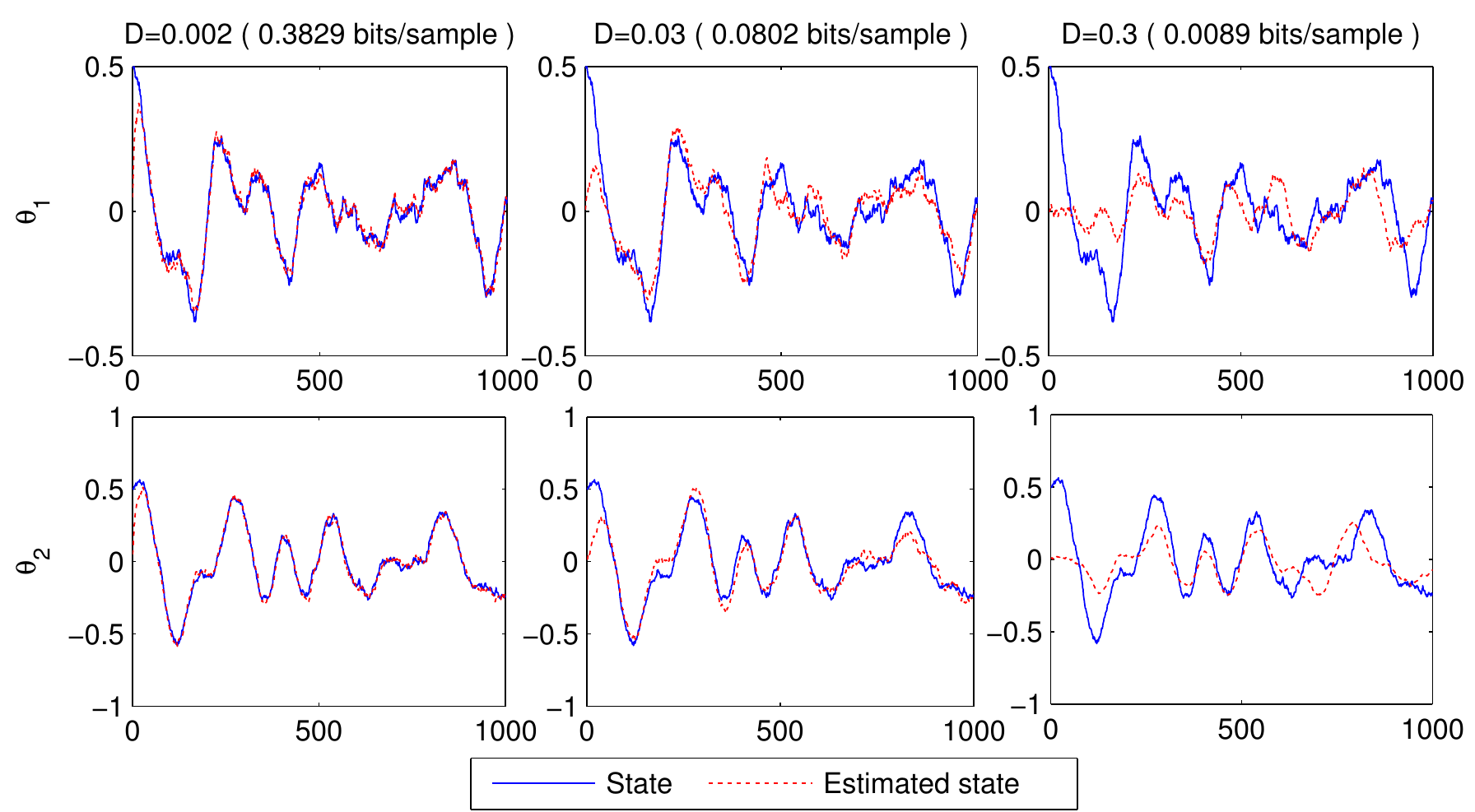} \fi
    \ifdefined\FIGSINGLECOL \includegraphics[width=.8\columnwidth]{pend2.pdf} \fi
    \caption{Tracking performance of the Kalman filter under different distortion constraints. (Tested on the same sample path of the noisy double pendulum.) }
    \label{fig:pend}
\end{figure}

\subsection{Optimal sensor design for double pendulum}
A linearized equation of motion of a double pendulum with friction and disturbance is given by
\[
\left[\!\! \begin{array}{c}
d\boldsymbol\theta_1 \\ d\boldsymbol\theta_2 \\ d\boldsymbol\omega_1 \\ d\boldsymbol\omega_2 \end{array}\!\! \right]\!\!=\!\!
\left[ \!\!\begin{array}{cccc} 0 \!\!\!&\!\!\! 0 \!\!\!&\!\!\! 1 \!\!\!&\!\!\! 0\\
0\!\!\!&\!\!\!0\!\!\!&\!\!\!0\!\!\!&\!\!\!1 \\
-\frac{(m_1+m_2)g}{m_1l_1} \!\!\!&\!\!\! \frac{m_2g}{m_1l_1} \!\!\!&\!\!\! -c_1 \!\!\!&\!\!\! 0\\
\frac{(m_1+m_2)g}{m_1l_2} \!\!\!&\!\!\! -\frac{(m_1+m_2)g}{m_1l_2} \!\!\!&\!\!\! 0 \!\!\!&\!\!\! -c_2\end{array}\!\!\right]\!\!
\left[\!\!\! \begin{array}{c}
\boldsymbol\theta_1 \\ \boldsymbol\theta_2 \\ \boldsymbol\omega_1 \\ \boldsymbol\omega_2 \end{array}\!\!\!\right]dt
 \!+\! d{\bf b}
\]
where ${\bf b}$ is a Brownian motion. 
We consider a discrete time model of the above equation of motion obtained through the Tustin transformation.
We are interested in designing a sensing model $\by_t=C\bx_t+\bv_t, \bv_t \sim \mathcal{N}(0,V)$ that optimally trades-off information cost and distortion level.\footnote{In practice, it is often the case that $\bx_t$ is partially observable through a given sensor mechanism. In such cases, the framework discussed in this paper is not appropriate. Instead, one can formulate an SRD problem for \emph{partially observable} Gauss-Markov processes. See \cite{tanaka2015zero} for details.}
We solve the stationary optimization problem (\ref{stationarylmi}) for this example with various values of $D$. The result is the sequential rate-distortion function shown in Figure \ref{fig:pend_srdf}. Finally, for every point on the trade-off curve, the optimal sensing matrices $C$ and $V$ are reconstructed, and the Kalman filter is designed base on them. Figure \ref{fig:pend} shows the trade-off between the distortion level and the tracking performance of the Kalman filter. When the distortion constraint is strict ($D=0.002$), the optimally designed sensor generates high rate information ($0.3829$ bits/sample) and the Kalman filter built on it tracks true state very well. When $D$ is large ($D=0.3$), the optimal sensing strategy chooses ``not to observe much", and the resulting Kalman filter shows poor tracking performance.

\subsection{Minimum down-link bandwidth for satellite attitude determination}
The equation of motion of the angular velocity vector of a spin-stabilized satellite linearized around the nominal angular velocity vector $(\omega_0, 0, 0)$ is
\[
\left[ \begin{array}{c} d\boldsymbol\omega_1 \\ d\boldsymbol\omega_2 \\ d\boldsymbol\omega_3 \end{array}\right]=
\left[ \begin{array}{ccc} 1 &0&0 \\
0 & 1 & \frac{I_3-I_1}{I_2}\omega_0 \\
0 & \frac{I_1-I_2}{I_3}\omega_0 & 1 \end{array}\right]
\left[ \begin{array}{c} \boldsymbol\omega_1 \\ \boldsymbol\omega_2 \\ \boldsymbol\omega_3 \end{array}\right]dt+d{\bf b}
\]
where ${\bf b}$ is a disturbance. 
Again, the equation of motion is converted to a discrete time model in the simulation.
Suppose that the satellite has on-board sensors that can accurately measure angular velocities, and the ground station needs to estimate them with some required accuracy (distortion) based on the transmitted data from the satellite.
Our interest is to determine the minimum down-link bit-rate that makes it possible, and identify what information needs to be transmitted to achieve this. 
Assume that the distortion constraints $D_t$ are time varying, but given \emph{a priori}.  (For instance, it must be kept small only when the satellite is in a mission.) 
The discussion so far indicates that the data to be transmitted is in the form of $\by_t=C_t \bx_t+\bv_t$ in order to minimize communication cost measured by $\sum_{t=1}^T I(\bx_t,\by_t|\by^{t-1})$.  
In Figure \ref{fig:sat}, a result of the SDP (\ref{optprob3}) is plotted, 
when the scheduling horizon is $T=120$ and a particular distortion constraint profile $D_t$ is given (shown in red in (a)).
The optimal down-link schedule shown in (b) requires no communication at all when the distortion constraint is met. As by-products of the SDP (\ref{optprob3}), the optimal scheduling of sensing matrices $C_t$ and noise covariances $V_t$ of $\bv_t$ can be also explicitly obtained.

\begin{figure}[t]
    \centering
    \ifdefined\FIGDOUBLECOL  \includegraphics[width=\columnwidth]{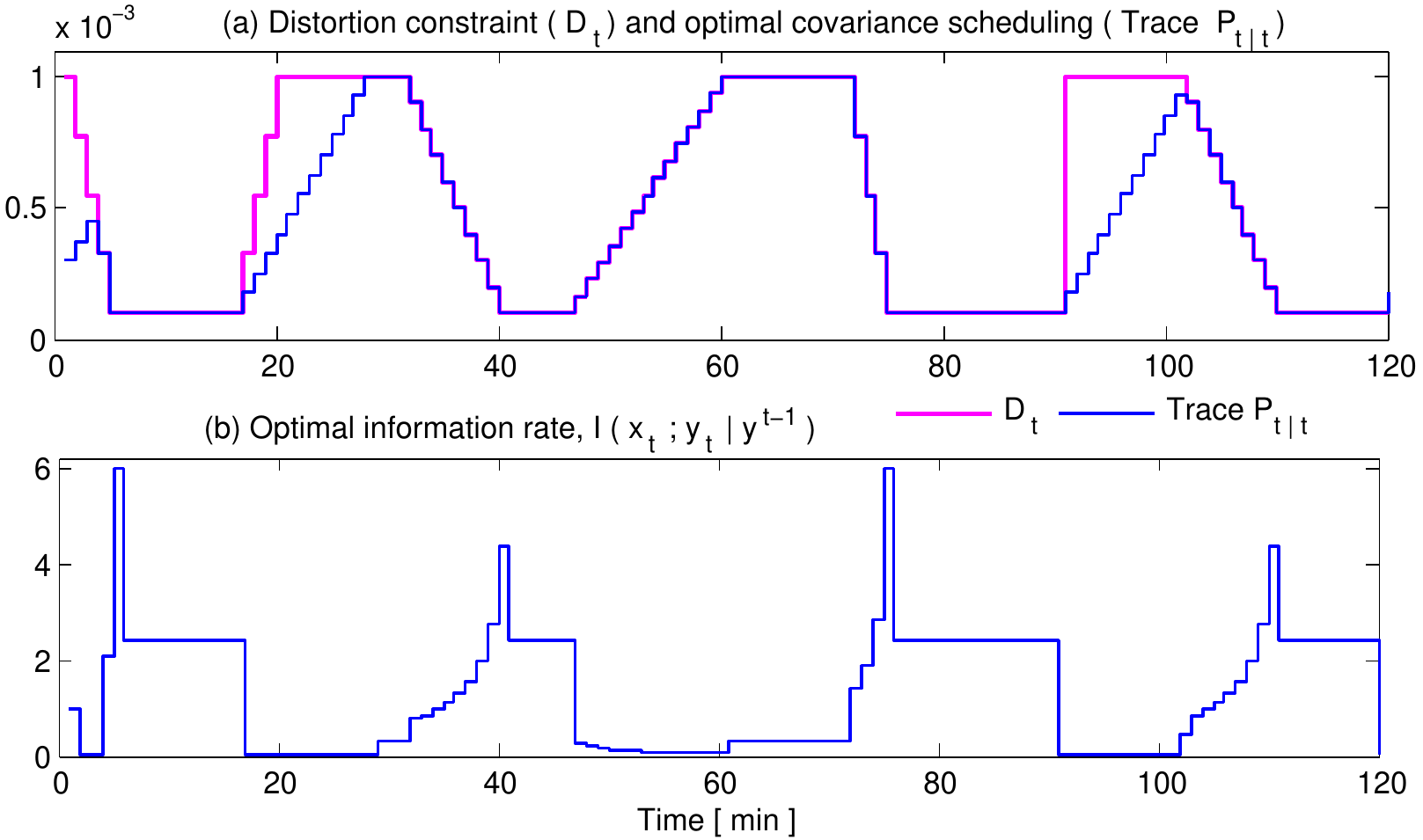} \fi
    \ifdefined\FIGSINGLECOL  \includegraphics[width=.8\columnwidth]{satellite3_crop.pdf} \fi
    \caption{Satellite attitude determination with time-varying distortion constraints. }
    \label{fig:sat}
\end{figure}

\section{Conclusion}
\label{secconclusion}
In this paper,  we revisited the ``sensor-estimator separation principle'' and showed that an optimal solution to the Gaussian SRD problem can be found by considering a related linear-Gaussian sensor design problem, which can be formulated as a determinant maximization problem with LMI constraints. The implication is that Gaussian SRD problems are efficiently solvable using standard SDP solvers.  We have also considered several potential applications of the Gaussian SRD problem and its relationship to real-time communication theory, networked control theory, and sensor scheduling problems.

% if have a single appendix:
%\appendix[Proof of the Zonklar Equations]
% or
%\appendix  % for no appendix heading
% do not use \section anymore after \appendix, only \section*
% is possibly needed

% use appendices with more than one appendix
% then use \section to start each appendix
% you must declare a \section before using any
% \subsection or using \label (\appendices by itself
% starts a section numbered zero.)
%

\appendices
\section{Mathematical preliminaries}
\label{appmathprelim}
\subsection{Stochastic kernels}
 Let $\mathcal{X}, \mathcal{Y}$ be Euclidean spaces. A \emph{(Borel-measurable) stochastic kernel} on $\mathcal{Y}$ given $\mathcal{X}$ is a map $q_{\by|\bx}: \mathcal{B_Y}\times \mathcal{X}\rightarrow [0,1]$ such that $q_{\by|\bx}(\cdot|x)$ is a probability measure on $(\mathcal{Y}, \mathcal{B_Y})$ for every $x\in\mathcal{X}$, and $q_{\by|\bx}(A|\cdot)$ is a Borel measurable function for every $A\in\mathcal{B_Y}$.
For simplicity, a stochastic kernel on $\mathcal{Y}$ given $\mathcal{X}$ will be denoted by $q(dy|x)$. The following results can be found in Propositions 7.27 and 7.28 in \cite{bertsekas1978stochastic}.
\begin{lemma}
\label{lemstochastickernel}
Let $\mathcal{X},\mathcal{Y}$ be Euclidean spaces.
\begin{itemize}[leftmargin=*]
\item[(a)] Let $r$ be a probability measure on $(\mathcal{X},\mathcal{B_X})$, and $q(dy|x)$ be a Borel measurable stochastic kernel on $\mathcal{Y}$ given $\mathcal{X}$. Then, there exists a unique probability measure $p$ on $(\mathcal{X}\times \mathcal{Y}, \mathcal{B}_{\mathcal{X}\times \mathcal{Y}})$ such that
\begin{equation}
\label{kerneldecomposition}
p(B_X \!\times \! B_Y)\!=\!\!\int_{B_X}\!\!\!\! q(B_Y|x)r(dx) \; \forall B_X\!\in \!\mathcal{B_X}, B_Y\!\in \!\mathcal{B_Y}.
\end{equation}
\item[(b)] Let $p$ be a probability measure on $(\mathcal{X}\times \mathcal{Y}, \mathcal{B}_{\mathcal{X}\times \mathcal{Y}})$. Then there exists a Borel-measurable stochastic kernel $q(dy|x)$ on $\mathcal{Y}$ given $\mathcal{X}$ such that (\ref{kerneldecomposition}) holds, where $r$ is the marginal of $p$ on $\mathcal{X}$.
\end{itemize}
\end{lemma}
Lemma \ref{lemstochastickernel} (a) guarantees the function $p$ defined on the algebra of measurable rectangles by (\ref{kerneldecomposition}) has a unique extension to the $\sigma$-algebra $\mathcal{B}_{\mathcal{X}\times\mathcal{Y}}$. For simplicity, the joint probability measure defined this way is denoted by
\begin{equation}
\label{kerneldecompositionsimple}
p(dx,dy)=q(dy|x)r(dx).
\end{equation}
Conversely, if the left hand side of (\ref{kerneldecompositionsimple}) is given, Lemma \ref{lemstochastickernel} (b) guarantees the existence of the decomposition on the right hand side.

\begin{definition}
\label{defcausal}
A stochastic kernel $q(dz^T|x^T)$ on $\mathcal{Z}^T$ given $\mathcal{X}^T$ is said to be  \emph{zero-delay} if it admits a factorization $q(dz^T|x^T)=\prod_{t=1}^T q(dz_t|z^{t-1},x^t)$.
\end{definition}
Once a zero-delay stochastic kernel is specified, successive applications of Lemma \ref{lemstochastickernel} (a) uniquely determine a joint probability measure by $q(dx^T,dz^T)=q(dx^T)\prod_{t=1}^T q(dz_t|z^{t-1},x^t)$.
The mutual information and the expectation in (P-SRD) is understood with respect to this joint probability measure.

Let $p$ and $q$ be probability measures on $\mathcal{X}=\mathbb{R}^n$. Whenever $p$ is absolutely continuous with respect to $q$ (denoted by $p \ll q$),  $\frac{dp}{dq}$ denotes the Radon-Nikodym derivative.
\begin{lemma}
\label{lemradon}
Let $\mathcal{X}$ and $\mathcal{Y}$ be Polish spaces.
\begin{itemize}[leftmargin=*]
\item[(a)] If $p,q,r$ are probability measures on $\mathcal{X}$ such that $r\ll q$ and $q\ll p$, then $r\ll p$ and $\frac{dr}{dp}=\frac{dr}{dq}\frac{dq}{dp} \; p-a.e.$. 
If $q \ll p$ and $p \ll q$, then $\frac{dp}{dq}\frac{dq}{dp}=1 \; a.e.$.
\item[(b)] Let $p_{\bx,\by}$ be a joint probability measure on $\mathcal{X}\times \mathcal{Y}$, and $p_\bx, p_\by$ be its marginals. Let $p_{\bx|\by}$ be a Borel-measurable stochastic kernel such that
\begin{equation}
\label{lemradon1}
p_{\bx,\by}(B_X\times B_Y)=\int_{B_Y} p_{\bx|\by}(B_X|y)p_\by(dy)
\end{equation}
for every $B_X\in \mathcal{B_X}, B_Y\in \mathcal{B_Y}$. If $p_{\bx,\by}\ll p_\bx \times p_\by$, then 
\begin{equation}
\frac{dp_{\bx,\by}}{d(p_\bx\times p_\by)}=\frac{dp_{\bx|\by}}{dp_\bx} \; p_\by - a.e..
\end{equation}
\end{itemize}
\end{lemma}
\begin{proof} For (a),  see Proposition 3.9 in \cite{folland1999real}. To prove (b), let $f(x,y)=\frac{dp_{\bx,\by}}{d(p_\bx\times p_\by)}$. By definition,
\begin{align*}
p_{\bx,\by}(B_X\times B_Y)&=\int_{B_X\times B_Y} f(x,y)(p_\bx\times p_\by)(dx,dy) \\
&=\int_{B_Y}\left(\int_{B_X} f(x,y)p_\bx(dx) \right)p_\by(dy)
\end{align*}
Since clearly $f\in L^1(p_\bx \times p_\by)$, the Fubini's theorem \cite{folland1999real} is applicable in the second line. Substituting this expression into (\ref{lemradon1}), we have
$\int_{B_X} f(x,y)p_\bx (dx)=p_{\bx|\by} (B_X|y) \; p_\by-a.e.$. Thus $f(x,y)=\frac{dp_{\bx|\by}}{dp_\bx} \; p_\by - a.e.$.
\end{proof}

\subsection{Information theoretic quantities}
The relative entropy, also known as the Kullback--Leibler divergence, from $p$ to $q$ is defined by
\[
D_{\text{KL}}(p\|q)=\begin{cases}
\int \log \frac{dp}{dq}dq & \text{if } p \ll q \\
+\infty & \text{otherwise.}
\end{cases}
\]
Relative entropy is always nonnegative.
Given two stochastic kernels $p_{\bx|\by}(dx|y)$ and $q_{\bx|\by}(dx|y)$ on $\mathcal{X}$ given $\mathcal{Y}$, and a probability measure $r_\by(dy)$, the conditional relative entropy is defined by
\[
D_{\text{KL}}(p_{\bx|\by}\|q_{\bx|\by} | r_\by)\!=\!\int_\mathcal{Y} D_{\text{KL}}(p_{\bx|\by}(dx|y) \| q_{\bx|\by}(dx|y)) r_\by(dy).
\]
Suppose $\mathcal{X}, \mathcal{Y}, \mathcal{Z}$ are Euclidean spaces, and $q_{\bx,\by}$ is a joint probability measure on $\mathcal{X}\times\mathcal{Y}$. Let $q_\bx$, $q_\by$ be its marginals, and $q_\bx \times q_\by$ be the product measure. The mutual information between $\bx$ and $\by$ is defined by $I(\bx;\by)=D_{KL}(q_{\bx,\by}||q_\bx\times q_\by)$.
Given a joint probability measure $q(dx,dy,dz)$, the conditional mutual information is defined by
\[
I(\bx;\by|\bz)=D_{\text{KL}}(q_{\bx,\by|\bz}\| q_{\bx|\by}\times q_{\by|\bz} | q_\bz).
\]
Suppose $\mathcal{X}=\mathbb{R}^n$, and $\bx$ is a $(\mathcal{X},\mathcal{B_X})$-valued random variable with probability measure $q_\bx$. Let $\lambda$ be the Lebesgue measure on $\mathcal{X}$ restricted to $\mathcal{B_X}$. The differential entropy of $\bx$ is defined by 
\[
h(\bx)=\begin{cases}
-\int \log \frac{dq_\bx}{d\lambda}dq_\bx & \text{if } q_\bx \ll \lambda \\
-\infty & \text{otherwise.}
\end{cases}
\]

\section{Proof of Lemma \ref{lemstep1}}
\label{secprooflemstep1}
(i):
Given a sequence of stochastic kernels $\gamma=\otimes_{t=1}^T q(dz_t|x^t, z^{t-1})\in \Gamma$ attaining cost $f_{\text{SRD}}<+\infty$ in (P-SRD), we are going to construct a sequence of linear-Gaussian stochastic kernels of the form (\ref{pi3lin}) that incurs no greater cost  than $f_{\text{SRD}}$ in (P-1).
Let $q(dx^T,dz^T)$ be the joint probability measure generated by $\gamma$ and the underlying Gauss-Markov process (\ref{gmprocess}). 
Without loss of generality, we can assume $q(dx^T,dz^T)$ has zero-mean. Otherwise, it is possible to choose an alternative feasible policy $\tilde{\gamma}\in\Gamma$ by linearly shifting $\gamma$ so that the resulting probability measure  $\tilde{q}(dx^T,dz^T)$ has zero-mean.
This operation does not increase the mutual information terms in the objective function.

Let $r(dx^T,dz^T)$ be a zero-mean, jointly Gaussian probability measure with the same covariance as $q(dx^T,dz^T)$. Let $E_t\bx_t+F_{t,t-1}\bz_{t-1}+\cdots+F_{t,1}\bz_1$ be the least mean square error estimate of $\bz_t$ given $\bx_t, \bz^{t-1}$ in $r(dx^T,dz^T)$, and let $\Gamma_t$ be the covariance matrix of the corresponding estimation error. 
Let $\{\bg_t\}$ be a sequence of Gaussian random vectors such that $\bg_t$ is independent of $\bx_0, \bw^t, \bg^{t-1}$ and $\bg_t\sim \mathcal{N}(0, \Gamma_t)$. For every $t=1,\cdots, T$, define a stochastic kernel $s(dz_t|x_t,z^{t-1})$ by
\[
\bz_t=E_t\bx_t+F_{t,t-1}\bz_{t-1}+\cdots+F_{t,1}\bz_1+\bg_t.
\]
We set $\gamma_1=\otimes_{t=1}^T s(dz_t|x_t,z^{t-1})\in \Gamma_1$ as a candidate solution to (P-1). 
By construction of $s(dz_t|x_t,z^{t-1})$, the following relation holds for every $t=1,\cdots, T$:
\begin{equation}
\label{kernelsr}
r(dx_t,dz^t)=s(dz_t|x_t,z^{t-1})r(dx_t,dz^{t-1}).
\end{equation}
Let $s(dx^T,dz^T)$ be a jointly Gaussian measure defined by $\{s(dz_t|x_t, z^{t-1}) \}_{t=1}^T$ and the process (\ref{gmprocess}). That is, it is a joint measure recursively defined by
\begin{subequations}
\begin{align}
& s(dx^t,dz^{t-1})=q(dx_t|x_{t-1})s(dx^{t-1}, dz^{t-1}) \label{sjoint1} \\
& s(dx^t,dz^t)=s(dz_t|x_t,z^{t-1})s(dx^t,dz^{t-1}). \label{sjoint2}
\end{align}
\end{subequations}
where $q(dx_t|x_{t-1})$ is a stochastic kernel defined by (\ref{gmprocess}).

Notice the following fact about $r(dx^T,dz^T)$.
\begin{proposition}\label{claim1}
For $t=2,\cdots, T$, let $r(dx_{t-1},dz^{t-1})$ and $r(dx_{t-1},dx_t,dz^{t-1})$ be marginals of $r(dx^T,dz^T)$. Then
 $r(dx_{t-1},dx_t,dz^{t-1})=q(dx_t|x_{t-1})r(dx_{t-1},dz^{t-1})$.
\end{proposition}
\begin{proof}
Since $\bz^{t-1}$ -- $\bx_{t-1}$ -- $\bx_t$ forms a Markov chain in the measure $q(dx^T,dz^T)$, by Lemma 3.2 of \cite{kim2010feedback}, $\bz^{t-1}$ -- $\bx_{t-1}$ -- $\bx_t$ forms a Markov chain under $r(dx^T,dz^T)$ as well.
Hence under $r$, $\bx_t$ is independent of $\bz^{t-1}$ given $\bx_{t-1}$, or $r(dx_t|x_{t-1},z^{t-1})=r(dx_t|x_{t-1})$. Moreover, since $q(dx_t,dx_{t-1})$ is a Gaussian distribution, and since $r$ is defined to be a Gaussian distribution with the same covariance as $q$, $r(dx_t,dx_{t-1})$ and $q(dx_t,dx_{t-1})$ have the same joint distribution. Hence, $q(dx_t|x_{t-1})=r(dx_t|x_{t-1})$. Thus, $r(dx_t|x_{t-1},z^{t-1})=q(dx_t|x_{t-1})$, proving the claim.
\end{proof}
In general, $r(dx^T,dz^T)$ and $s(dx^T,dz^T)$ are different joint probability measures. However, we have the following result.
\begin{proposition}\label{claim2}
For every $t=1,\cdots,T$, let $r(dx_t,dz^t)$ and $s(dx_t,dz^t)$ be marginals of $r(dx^T,dz^T)$ and $s(dx^T,dz^T)$ respectively. Then $r(dx_t,dz^t)=s(dx_t,dz^t)$.
\end{proposition}
\begin{proof}
By definitions,
\begin{align*}
&r(dx_1,dz_1)=s(dz_1|x_1)r(dx_1) \\
&s(dx_1,dz_1)=s(dz_1|x_1)q(dx_1).
\end{align*}
Since $r(dx_1)=q(dx_1)$, $r(dx_1,dz_1)=s(dx_1,dz_1)$ holds.
So assume that the claim holds for $t=k-1$. Then
\ifdefined\FIGDOUBLECOL
\begin{subequations}
\begin{align}
&s(dx_k, dz^k) \nonumber  \\
&=s(dz_k|x_k,z^{k-1}) s(dx_k,dz^{k\!-\!1}) \label{preq0} \\
&=s(dz_k|x_k,z^{k-1}) \!\int_{\mathcal{X}_{k-1}} \!\!\!\!\!\!\! s(dx_{k-1}, dx_k, dz^{k-1})  \nonumber \\
&=s(dz_k|x_k,z^{k-1}) \!\int_{\mathcal{X}_{k-1}} \!\!\!\!\!\!\! q(dx_k|x_{k-1}) s(dx_{k-1}, dz^{k-1})  \label{preq1} \\
&=s(dz_k|x_k,z^{k-1}) \!\int_{\mathcal{X}_{k-1}} \!\!\!\!\!\!\! q(dx_k|x_{k-1}) r(dx_{k-1},dz^{k-1})  \label{preq2} \\
&=s(dz_k|x_k,z^{k-1}) \!\int_{\mathcal{X}_{k-1}} \!\!\!\!\!\!\! r(dx_{k-1}, dx_k, dz^{k-1})  \label{preq3} \\
&=s(dz_k|x_k,z^{k-1}) r(dx_k, dz^{k-1}) \nonumber \\
&=r(dx_k, dz^k). \label{preq4}
\end{align}
\end{subequations}
\fi
\ifdefined\FIGSINGLECOL 
\begin{subequations}
\begin{align}
s(dx_k, z^k) &=s(dz_k|x_k,z^{k-1}) s(dx_k,dz^{k\!-\!1}) \label{preq0} \\
&=s(dz_k|x_k,z^{k-1}) \!\int_{\mathcal{X}_{k-1}} \!\!\!\!\!\!\! s(dx_{k-1}, dx_k, dz^{k-1})  \nonumber \\
&=s(dz_k|x_k,z^{k-1}) \!\int_{\mathcal{X}_{k-1}} \!\!\!\!\!\!\! q(dx_k|x_{k-1}) s(dx_{k-1}, dz^{k-1})  \label{preq1} \\
&=s(dz_k|x_k,z^{k-1}) \!\int_{\mathcal{X}_{k-1}} \!\!\!\!\!\!\! q(dx_k|x_{k-1}) r(dx_{k-1},dz^{k-1})  \label{preq2} \\
&=s(dz_k|x_k,z^{k-1}) \!\int_{\mathcal{X}_{k-1}} \!\!\!\!\!\!\! r(dx_{k-1}, dx_k, dz^{k-1})  \label{preq3} \\
&=s(dz_k|x_k,z^{k-1}) r(dx_k, dz^{k-1}) \nonumber \\
&=r(dx_k, dz^k). \label{preq4}
\end{align}
\end{subequations}
\fi
The first step (\ref{preq0}) follows from the definition (\ref{sjoint2}).
Step (\ref{preq1}) also follows from the definition (\ref{sjoint1}).
In (\ref{preq2}), the induction assumption $s(dx_{k-1},dz^{k-1})=r(dx_{k-1},dz^{k-1})$ was used.
The result of Proposition~\ref{claim1} was used in (\ref{preq3}).
The final step (\ref{preq4}) is due to (\ref{kernelsr}).
\end{proof}

To prove that $\gamma_1=\otimes_{t=1}^T s(dz_t|x_t,z^{t-1})$ incurs no greater cost than $f_{\text{SRD}}$ in (P-1), notice that replacing $q(dz_t|x^t,z^{t-1})$ with $s(dz_t|x_t,z^{t-1})$ will not change the distortion:
\begin{align}
\mathbb{E}_q \|\bx_t-\bz_t\|_{\Theta_t}^2&= \int \|x_t-z_t\|_{\Theta_t}^2 q(dx_t, dz^t) \nonumber \\
&= \int \|x_t-z_t\|_{\Theta_t}^2 r(dx_t, dz^t) \label{pexpeq1}  \\
&= \int \|x_t-z_t\|_{\Theta_t}^2 s(dx_t, dz^t) \label{pexpeq2} \\
&= \mathbb{E}_s \|\bx_t-\bz_t\|_{\Theta_t}^2. \nonumber
\end{align}
Equality (\ref{pexpeq1}) holds since $q$ and $r$ have the same second order properties. The result of Proposition~\ref{claim2} was used in step (\ref{pexpeq2}). 

Next, we show that the mutual information never increases by this replacement. 
\begin{proposition}\label{claim3} If $I_q(\bx^T;\bz^T) <+\infty$, then
$I_{r}(\bx^T;\bz^T) \leq I_q(\bx^T;\bz^T)$. 
\end{proposition}
\begin{proof}
This can be directly verified as
\ifdefined\FIGDOUBLECOL
\begin{align}
&I_q(\bx^T;\bz^T) - I_{r}(\bx^T;\bz^T) \nonumber \\
&=\int \log\frac{dq(x^T|z^T)}{dq(x^T)}q(dx^T,dz^T) \label{kernelproof1} \\
&  \;\;\;\;\;\; -\int \log\frac{dr(x^T|z^T)}{dr(x^T)}r(dx^T,dz^T) \label{kernelproof2}\\
&=\int \log\frac{dq(x^T|z^T)}{dq(x^T)}q(dx^T,dz^T) \nonumber \\
&  \;\;\;\;\;\; -\int \log\frac{dr(x^T|z^T)}{dr(x^T)}q(dx^T,dz^T) \label{kernelproof3}\\
&=\int \log \left( \frac{dq(x^T|z^T)}{dq(x^T)} \cdot \frac{dr(x^T)}{dr(x^T|z^T)} \right) q(dx^T,dz^T) \nonumber  \\
&=\int \log \left( \frac{dq(x^T|z^T)}{dr(x^T|z^T)} \right) q(dx^T,dz^T) \label{kernelproof4} \\
&=\int \left( \int \log \left( \frac{dq(x^T|z^T)}{dr(x^T|z^T)} \right) q(dx^T|z^T)  \right)  q(dz^T) \nonumber \\
&=\int D_{\text{KL}}\left(q(x^T|z^T) || r(x^T|z^T)\right)  q(dz^T) \geq 0.\nonumber
\end{align} 
\fi
\ifdefined\FIGSINGLECOL 
\begin{align}
I_q(\bx^T;\bz^T) - I_{r}(\bx^T;\bz^T) 
&=\int \log\frac{dq(x^T|z^T)}{dq(x^T)}q(dx^T,dz^T) \label{kernelproof1} \\
&  \;\;\;\;\;\; -\int \log\frac{dr(x^T|z^T)}{dr(x^T)}r(dx^T,dz^T) \label{kernelproof2}\\
&=\int \log\frac{dq(x^T|z^T)}{dq(x^T)}q(dx^T,dz^T) \nonumber \\
&  \;\;\;\;\;\; -\int \log\frac{dr(x^T|z^T)}{dr(x^T)}q(dx^T,dz^T) \label{kernelproof3}\\
&=\int \log \left( \frac{dq(x^T|z^T)}{dq(x^T)} \cdot \frac{dr(x^T)}{dr(x^T|z^T)} \right) q(dx^T,dz^T) \nonumber  \\
&=\int \log \left( \frac{dq(x^T|z^T)}{dr(x^T|z^T)} \right) q(dx^T,dz^T) \label{kernelproof4} \\
&=\int \left( \int \log \left( \frac{dq(x^T|z^T)}{dr(x^T|z^T)} \right) q(dx^T|z^T)  \right)  q(dz^T) \nonumber \\
&=\int D_{\text{KL}}\left(q(x^T|z^T) || r(x^T|z^T)\right)  q(dz^T) \geq 0.\nonumber
\end{align}  
\fi
(\ref{kernelproof1}) is by definition of mutual information and Lemma \ref{lemradon} (b). Since $q(dx^T)$ is a non-degenerate Gaussian probability measure, $I_q(\bx^T;\bz^T) <+\infty$ implies that $q(dx^T|z^T)$ admits a density $q(dz^T)-a.e.$.
This further requires that a Gaussian measure $r(dx^T|z^T)$ admits a density everywhere in $\text{supp}(r(dz^T))$, i.e., the support of the probability measure $r(dz^T)$.
Thus, the Radon-Nikodym derivative in (\ref{kernelproof2}) exists everywhere in $\text{supp}(r(dz^T))$.
Since $r$ is a Gaussian probability measure, $\log\frac{dr(x^T|z^T)}{dr(x^T)}$ is a quadratic function of $x^T$ and $z^T$ everywhere in $\text{supp}(r(dx^T, dz^T))$.
Since it can be shown that $\text{supp}(q(dx^T, dz^T))\subseteq \text{supp}(r(dx^T, dz^T))$, this allows us to replace $r(dx^T,dz^T)$ with $q(dx^T,dz^T)$ in  (\ref{kernelproof3}) since they have the same second order moments.
Lemma \ref{lemradon1} (a) is applicable in  (\ref{kernelproof4}) since $r(dx^T)=q(dx^T)$.
\end{proof}
Finally,  
\begin{align}
\sum\nolimits_{t=1}^T I_q(\bx^t;\bz_t|\bz^{t-1}) &= I_q(\bx^T;\bz^T) \label{ipiqir1}\\
&\geq I_{r}(\bx^T;\bz^T) \label{ipiqir2}\\
&= \sum\nolimits_{t=1}^T I_{r}(\bx^T;\bz_t|\bz^{t-1}) \nonumber \\
&\geq \sum\nolimits_{t=1}^T I_{r}(\bx_t;\bz_t|\bz^{t-1}) \nonumber \\
&= \sum\nolimits_{t=1}^T I_{s}(\bx_t;\bz_t|\bz^{t-1}) \label{ipiqir3}
\end{align}
See Remark~\ref{remark_directedinfo} for the equality (\ref{ipiqir1}). The result of Proposition~\ref{claim3} was used in (\ref{ipiqir2}). Equality (\ref{ipiqir3}) follows from Proposition~\ref{claim2}.
Thus, using $\gamma=\otimes_{t=1}^T q(dz_t|x^t, z^{t-1})\in \Gamma$ attaining cost $f_{\text{SRD}}$ in (P-SRD), we have constructed $\gamma_1=\otimes_{t=1}^T s(dz_t|x_t, z^{t-1})\in \Gamma_1$ incurring smaller cost in (P-1) than $f_{\text{SRD}}$.

(ii):
Let $\gamma_1=\otimes_{t=1}^T q(dz_t|x_t,z^{t-1})\in \Gamma_1$ be a sequence of linear-Gaussian stochastic kernels attaining $f_1<+\infty$ in (P-1), and $q(dx^T,dz^T)$ be the resulting joint probability measure.
Since $\bz_t$ -- $(\bx_t,\bz^{t-1})$ -- $\bx^{t-1}$ forms a Markov chain in  $q(dx^T,dz^T)$, we have
\ifdefined\FIGDOUBLECOL
\begin{align}
I(\bx^t;\bz_t|\bz^{t-1})&=I(\bx_t;\bz_t|\bz^{t-1})+I(\bx^{t-1};\bz_t|\bx_t,\bz^{t-1}) \nonumber \\
&= I(\bx_t;\bz_t|\bz^{t-1}). \label{eqixt}
\end{align}
\fi
\ifdefined\FIGSINGLECOL  
\begin{equation}
I(\bx^t;\bz_t|\bz^{t-1})=I(\bx_t;\bz_t|\bz^{t-1})+I(\bx^{t-1};\bz_t|\bx_t,\bz^{t-1}) = I(\bx_t;\bz_t|\bz^{t-1}). \label{eqixt}
\end{equation}
\fi
Hence the mutual information terms in (P-1) can be replaced with the ones in (P-SRD) without increasing cost.

\section{Proof of Lemma \ref{lemstep2}}
\label{secprooflemstep2}
(i): Suppose 
\begin{equation}
\label{lemprEF} \bz_t\!=\!E_t \bx_t \!+\! F_{t,t-1} \bz_{t-1} \!+\! \cdots \!+\! F_{t,1} \bz_1 \!+\! \bg_t, t\!=\!1,\cdots\!, T 
\end{equation}
is a linear-Gaussian stochastic kernel that attains $f_1<+\infty$ in (P-1). It is sufficient for us to show that there exist  nonnegative integers $r_1,\cdots, r_T$ and matrices $C_t\in \mathbb{R}^{r_t\times n_t}, V_t\in \mathbb{S}_{++}^{r_t}, t=1,\cdots, T$ such that $\{C_t, V_t\}_{t=1}^T$ attains a smaller cost than $f_1$ in (P-LGS). Let
\[ \left[\begin{array}{cc}U_1 & U_2\end{array}\right] 
\left[\begin{array}{cc}\Sigma_1 & 0 \\ 0 & 0\end{array}\right] 
\left[\begin{array}{c} U_1^\top \\ U_2^\top \end{array}\right] 
=\mathbb{E}\bg_t\bg_t^\top
\]
with an orthonormal matrix $U=\left[\begin{array}{cc}U_1 & U_2\end{array}\right] $ be a singular value decomposition of the covariance matrix of $\bg_t$.  If $\bg_t$ is nondegenerate, we understand that $U=U_1$, while if $\bg_t$ is a point mass at zero, then $U=U_2$.
Clearly $\tilde{\bg}_t=U_1^\top \bg_t$ is a zero-mean, nondegenerate Gaussian random vector and $U_2^\top \bg_t=0$. Define
\[
\left[\!\!\begin{array}{c} \tilde{\bz}_t \\ \hat{\bz}_t \end{array}\!\!\right]\!=\!
\left[\!\!\begin{array}{c} U_1^\top \\ U_2^\top \end{array}\!\!\right]\bz_t, 
\left[\!\!\begin{array}{c} \tilde{E}_t \\ \hat{E}_t \end{array}\!\!\right]\!=\!
\left[\!\!\begin{array}{c} U_1^\top \\ U_2^\top \end{array}\!\!\right]E_t, 
\left[\!\!\begin{array}{c} \tilde{F}_{t,s} \\ \hat{F}_{t,s} \end{array}\!\!\right]\!=\!
\left[\!\!\begin{array}{c} U_1^\top \\ U_2^\top \end{array}\!\!\right]F_{t,s}
\]
for $s=1,\cdots, t-1$. Then multiplying (\ref{lemprEF}) by $U^\top$ from the left yields
\begin{equation}
\label{eqztildehat}
\left[\!\!\!\begin{array}{c} \tilde{\bz}_t \\ \hat{\bz}_t \end{array}\!\!\!\right]\!=\!
\left[\!\!\!\begin{array}{c} \tilde{E}_t \\ \hat{E}_t \end{array}\!\!\!\right]\bx_t\!+\!
\left[\!\!\!\begin{array}{c} \tilde{F}_{t,t-1} \\ \hat{F}_{t,t-1} \end{array}\!\!\!\right]\bz_{t-1}\!+\!\cdots\!+\!
\left[\!\!\!\begin{array}{c} \tilde{F}_{t,1} \\ \hat{F}_{t,1} \end{array}\!\!\!\right]\bz_1\!+\!
\left[\!\!\!\begin{array}{c} \tilde{\bg}_t \\ 0\end{array}\!\!\!\right].
\end{equation}
\begin{proposition}
\label{claimehat}
$\hat{E}_t=0 \; \forall t=1,\cdots, T$ is necessary for $f_1 < +\infty$.
\end{proposition}
\begin{proof}
Focus on the mutual information terms in (P-1).
\ifdefined\FIGDOUBLECOL
\begin{align*}
&I(\bx_t;\bz_t|\bz^{t-1}) \\
&=I(\bx_t;\tilde{\bz}_t, \hat{\bz}_t |\bz^{t-1}) \\
&\geq I(\bx_t;\hat{\bz}_t |\bz^{t-1}) \\
&=I(\bx_t;\hat{E}_t\bx_t+\hat{F}_{t,t-1}\bz_{t-1}+\cdots+\hat{F}_{t,1}\bz_{1} |\bz^{t-1}) \\
&=I(\bx_t;\hat{E}_t\bx_t |\bz^{t-1}) \\
&=I(\bx_t;\hat{E}_t\bx_t, \bz^{t-1})-I(\bx_t;\bz^{t-1}) \\
&\geq I(\bx_t;\hat{E}_t\bx_t)-I(\bx_t;\bz^{t-1}) 
\end{align*}
\fi
\ifdefined\FIGSINGLECOL  
\begin{align*}
I(\bx_t;\bz_t|\bz^{t-1}) &=I(\bx_t;\tilde{\bz}_t, \hat{\bz}_t |\bz^{t-1}) \\
&\geq I(\bx_t;\hat{\bz}_t |\bz^{t-1}) \\
&=I(\bx_t;\hat{E}_t\bx_t+\hat{F}_{t,t-1}\bz_{t-1}+\cdots+\hat{F}_{t,1}\bz_{1} |\bz^{t-1}) \\
&=I(\bx_t;\hat{E}_t\bx_t |\bz^{t-1}) \\
&=I(\bx_t;\hat{E}_t\bx_t, \bz^{t-1})-I(\bx_t;\bz^{t-1}) \\
&\geq I(\bx_t;\hat{E}_t\bx_t)-I(\bx_t;\bz^{t-1}) 
\end{align*}
\fi
Recall that $\bx_t$ is defined by (\ref{gmprocess}) and is a nondegenerate Gaussian random vector. If $\hat{E}_t\bx_t$ is a non-zero linear function of $\bx_t$, then $I(\bx_t;\hat{E}_t\bx_t)=+\infty$, while $I(\bx_t;\bz^{t-1}) $ is bounded. Therefore, $\hat{E}_t=0$ is necessary for $I(\bx_t;\bz_t|\bz^{t-1})$ to be bounded.
\end{proof}
Proposition~\ref{claimehat}, together with (\ref{eqztildehat}), implies that $\hat{\bz}_t$ is a linear function of $\bz^{t-1}$. Hence, there exist some matrices $H_{t,1},\cdots, H_{t,t-1}$ such that the first row of (\ref{eqztildehat}) can be rewritten as
\begin{equation}
\label{eqzhat}
\tilde{\bz}_t=\tilde{E}_t \bx_t+H_{t,t-1}\tilde{\bz}_{t-1}+\cdots+H_{t,1}\tilde{\bz}_{1}+\tilde{\bg}_t.
\end{equation}
It is also easy to see that $\bz^t$ can be fully reconstructed if $\tilde{\bz}^t$ is given. In particular, this implies that the $\sigma$-algebras generated by $\bz^t$ and $\tilde{\bz}^t$  are the same.
\begin{equation}
\label{sigmaz}
\sigma(\bz^t)=\sigma(\tilde{\bz}^t).
\end{equation}
\begin{proposition}
\label{claimiztilde}
$I(\bx_t;\bz_t|\bz^{t-1})=I(\bx_t;\tilde{\bz}_t|\tilde{\bz}^{t-1})\; \forall t=1,\cdots,T$.
\end{proposition}
\begin{proof} This can be directly verified as follows.
\begin{align}
I(\bx_t;\bz_t|\bz^{t-1})&=I(\bx_t;\tilde{\bz}_t, \hat{\bz}_t|\bz^{t-1}) \nonumber \\
&= I(\bx_t;\tilde{\bz}_t|\bz^{t-1}) \label{lemprclaim1}\\
&= I(\bx_t;\tilde{\bz}_t|\tilde{\bz}_{t-1},\hat{\bz}_{t-1},\bz^{t-2}) \nonumber \\
&= I(\bx_t;\tilde{\bz}_t|\tilde{\bz}_{t-1},\bz^{t-2}) \label{lemprclaim2} \\
&= I(\bx_t;\tilde{\bz}_t|\tilde{\bz}_{t-1},\tilde{\bz}_{t-2},\bz^{t-3}) \nonumber \\
& \hspace{5ex} \vdots \nonumber \\
&= I(\bx_t;\tilde{\bz}_t|\tilde{\bz}^{t-1}) \nonumber
\end{align}
Equality (\ref{lemprclaim1}) holds since $\hat{\bz}_t$ is a linear function of $\bz^{t-1}$. Similarly, (\ref{lemprclaim2}) holds because $\hat{\bz}_{t-1}$ is a linear function of $\bz^{t-2}$. The remaining equalities can be shown by repeating the same argument.
\end{proof}
Now, for every $t=1,\cdots, T$, set $C_t=\tilde{E}_t$ and $\bv_t=\tilde{\bg}_t$. Then, by construction, $\bv_t$ is a zero-mean, nondegenerate Gaussian random vector that is independent of $\bx_0, \bw^t, \bv^{t-1}$. Hence $\by_t=C_t\bx_t+\bv_t$ is an admissible sensor equation for (P-LGS).
\begin{proposition}
\label{claimiztildey}
$I(\bx_t;\by_t|\by^{t-1})=I(\bx_t;\tilde{\bz}_t|\tilde{\bz}^{t-1})\; \forall t=1,\cdots,T$.
\end{proposition}
\begin{proof}
By concatenating (\ref{eqzhat}), it can be easily seen that an identity $\mathcal{H}_t \tilde{\bz}^t=\by^t$ holds for every $t=1,\cdots, T$, where $\mathcal{H}_t$ is an invertible matrix defined by
\[ \mathcal{H}_t = \left[ \begin{array}{cccc} I & 0 & \cdots & 0\\ 
-H_{2,1} & I & \ddots & \vdots  \\
\vdots & \ddots & \ddots & 0 \\
-H_{t,1} & \cdots & -H_{t,t-1} & I\end{array}\right]. \]
Hence,
\ifdefined\FIGDOUBLECOL
\begin{align*}
I(\bx_t;\by_t|\by^{t-1}) &= I(\bx_t; \by_t|\tilde{\bz}^{t-1}) \\
&=\! I(\bx_t;\tilde{\bz}_t\!-\!H_{t,t-1} \tilde{\bz}_{t-1} \!-\! \cdots \!-\! H_{t,1} \tilde{\bz}_1|\tilde{\bz}^{t-1}) \\
&=\! I(\bx_t;\tilde{\bz}_t|\tilde{\bz}^{t-1}).
\end{align*}
\fi
\ifdefined\FIGSINGLECOL  
\begin{align*}
I(\bx_t;\by_t|\by^{t-1}) &= I(\bx_t; \by_t|\tilde{\bz}^{t-1}) \\
&= I(\bx_t;\tilde{\bz}_t-H_{t,t-1} \tilde{\bz}_{t-1} - \cdots - H_{t,1} \tilde{\bz}_1|\tilde{\bz}^{t-1}) \\
&= I(\bx_t;\tilde{\bz}_t|\tilde{\bz}^{t-1}).
\end{align*}
\fi
\end{proof}
Thus, starting from a sequence of linear-Gaussian stochastic kernels (\ref{lemprEF}), we have constructed a sequence of sensor equations of the form $\by_t=C_t\bx_t+\bv_t$ such that $I(\bx_t;\by_t|\by^{t-1})=I(\bx_t;\bz_t|\bz^{t-1})$. The last equality is a consequence of Propositions~\ref{claimiztilde} and \ref{claimiztildey}. To complete the proof of the first statement of Lemma \ref{lemstep2}, it is left to show that 
\begin{equation}
\label{lmsee}
\mathbb{E}\|\bx_t-\bz_t'\|_{\Theta_t}^2 \leq \mathbb{E}\|\bx_t-\bz_t\|_{\Theta_t}^2 \; \forall t=1,\cdots, T
\end{equation}
where $\bz_t'=\mathbb{E}(\bx_t|\by^t)$. (Here, we refer to the variable ``$\bz_t$'' in  (P-LGS) as $\bz_t'$ in order to distinguish it from the variable $\bz_t$ in (P-1).)
The inequality (\ref{lmsee}) can be verified by the following observation. Since $\mathcal{H}_t\tilde{\bz}^t=\by^t$, we have $\sigma(\by^t)=\sigma(\tilde{\bz}^t)$. Moreover, it follows from (\ref{sigmaz}) that $\sigma(\by^t)=\sigma(\bz^t)$. Thus, $\bz_t$ is $\sigma(\by^t)$-measurable. However, since $\bz_t'=\mathbb{E}(\bx_t|\by^t)$, $\bz_t'$ minimizes the mean square estimation error among all $\sigma(\by^t)$-measurable functions. Thus (\ref{lmsee}) must hold.

(ii):
Let $\{C_t, V_t\}_{t=1}^T$ be a sequence of matrices that attains $f_{\text{LGS}}<+\infty$ in (P-LGS). Let $\by_t$ be defined by (\ref{srdchannel}), and $\bz_t'=\mathbb{E}(\bx_t|\by^t)$ be the least mean square error estimate of $\bx_t$ given $\by^t$ obtained by the Kalman filter.
From the Kalman filtering formula, we have
\begin{align*}
\bz_t'&\!=\!A_{t\!-\!1}\bz'_{t\!-\!1}\!\!+\!P_{t|t\!-\!1}C_t^\top (C_tP_{t|t\!-\!1}C_t^\top \!\!+\!\!V_t)^{-1}(\by_t\!-\!C_tA_{t\!-\!1}\bz'_{t\!-\!1}) \\
&\!=\! E_t \bx_t + F_{t,t-1} \bz'_{t-1}+ \cdots + F_{t,1} \bz'_{1}+ \bg_t
\end{align*}
where $E_t, F_{t,t-1}, \cdots, F_{t,1}$ are some matrices (in fact, all $F_{t,t-2}, \cdots, F_{t,1}$ are zero matrices) and $\bg_t$ is a zero-mean Gaussian random vector that is independent of $\bx_0, \bw^t$ and $\bg^{t-1}$. Hence, by constructing a linear-Gaussian stochastic kernel for (P-1) by
\[\bz_t= E_t \bx_t + F_{t,t-1} \bz_{t-1}+ \cdots + F_{t,1} \bz_{1}+ \bg_t \]
using the same $E_t, F_{t,t-1}, \cdots, F_{t,1}$ and $\bg_t$, $(\bx^T, \bz^T)$ and $(\bx^T, \bz'^T)$ have the same joint distribution. 
Thus $\mathbb{E}\|\bx_t-\bz_t'\|_{\Theta_t}^2 = \mathbb{E}\|\bx_t-\bz_t\|_{\Theta_t}^2 \; \forall t=1,\cdots, T$.
Hence, it remains to prove that
\[ I(\bx_t;\by_t|\by^{t-1}) \geq I(\bx_t;\bz_t|\bz^{t-1}) \;\; \forall t=1,\cdots, T. \]
Notice that $I(\bx_t;\by_t|\bz^{t-1}) \geq I(\bx_t;\bz_t|\bz^{t-1})$  is immediate from the data-processing inequality. Moreover, an equality $I(\bx_t;\by_t|\by^{t-1}) = I(\bx_t;\by_t|\bz^{t-1})$ holds since the input sequence $\by^{t-1}$ and the output sequence $\bz^{t-1}$ of the Kalman filter contain statistically equivalent information.
Formally, this can be shown by proving that the Kalman filter is causally invertible \cite{kailath1968innovations}, and thus one can construct $\by^{t-1}$ from $\bz^{t-1}$ and \emph{vice versa}.

\ifdefined\LONGVERSION
\section{Semidefinite representation of SRD problems}
\label{secsdr}

We begin with the observation that the set
$
K^{2^l}=\left\{(x_1,\cdots, x_{2^l}, y)\in \mathbb{R}_+^{2^l+1}: y \leq \left( x_1\cdots x_{2^l} \right)^{\frac{1}{2^l}} \right\}
$
is semidefinite representable. To see this, introduce $1+2+2^2+\cdots+2^{l-1}=2^l-1$ dimensional vector
$ u= (u_{1,1},\cdots, u_{1,2^{l-1}}, u_{2,1},\cdots, u_{1,2^{l-2}}, \cdots, u_{l,1})$
and consider the following multiple layers of constraints:
\begin{itemize}
\item Layer $1$:
$ \left[\begin{array}{cc}
x_{2i-1} & u_{1,i} \\ u_{1,i} & x_{2i}
\end{array}\right] \succeq 0 \;\; \forall i=1,\cdots, 2^{l-1},
 $
\item Layer $j$ ($2\leq j \leq l$):
$
\left[\begin{array}{cc}
u_{j-1, 2i-1} & u_{j,i} \\ u_{j,i} & u_{j-1, 2i}
\end{array}\right] \succeq 0 \;\; \forall i=1,\cdots, 2^{l-j},
$
\item Layer $l+1$:
$ u_{l,1}-t \geq 0, \;\; t \geq 0.$
\end{itemize}
The above collection of constraints is simply denoted by  $\mathcal{S}_l(x;y;u)\succeq 0$, where $\mathcal{S}_l(x;y;u)$ is a $2^l+2$ dimensional block diagonal symmetric matrix such that each of its diagonal block corresponds to one of the above LMIs, and hence $\mathcal{S}_l(\cdot)$ is linear with respect to its arguments. It is not difficult to prove the following fact, which can be found in Section 4.2 in \cite{ben2001}.
\begin{lemma}\label{lemKsdr} $K^{2^l}$ is semidefinite representable as
\[
K^{2^l}=\left\{ \begin{array}{l} (x_1,\cdots, x_{2^l}, y)\in \mathbb{R}_+^{2^l+1}: \\
\exists u \in \mathbb{R}^{2^l-1} \text{ s.t. } \mathcal{S}_l(x;y;u)\succeq 0 \end{array}
 \right\}. 
\]
\end{lemma}
Lemma \ref{lemKsdr} immediately leads to a semidefinite representation of several convex rational functions. In particular, we will use the following facts.
\begin{itemize}
\item Let $m$ be a positive integer. Then the set
$
K^{m}=\left\{(x, y)\in \mathbb{R}_+^m\times \mathbb{R}_+ : y \leq \left( x_1\cdots x_m \right)^{\frac{1}{m}} \right\}
$
is semidefinite representable. To obtain a semidefinite representation, set $l=\lceil \log_2 m \rceil$ (i.e., $l$ is the smallest integer such that $2^l \geq m$) and define a linear map 
\[\mathcal{Q}_m(x;y;u)=\mathcal{S}_l(x_1,\cdots, x_m, \underbrace{y, \cdots, y}_{2^l-m \text{ copies } }; y; u). \]
Then it can be shown that 
\[K^{m}=\left\{\begin{array}{l} (x,y)\in \mathbb{R}_+^m\times \mathbb{R}_+: \\ 
\exists u \in \mathbb{R}^{2^l-1} \text{ s.t. } \mathcal{Q}_m(x;y;u)\succeq 0  \end{array} \right\}.
\]
\item Let $m$ be a positive integer. Then the set
$
L^{m}=\left\{(x, y)\in \mathbb{R}_+^m\times \mathbb{R}_+ : y \geq \left( x_1\cdots x_m \right)^{-1} \right\}
$
is semidefinite representable. Define $l=\lceil \log_2 (m+1) \rceil$ and a linear map
\[\mathcal{R}_m(x;y;u)=\mathcal{S}_l(x_1,\cdots, x_m, y, \underbrace{1, \cdots, 1}_{2^l-m-1 \text{ copies } }; 1; u). \]
Then it can be shown that 
\[
L^{m}=\left\{\begin{array}{l} (x,y)\in \mathbb{R}_+^m\times \mathbb{R}_+: \\
\exists u \in \mathbb{R}^{2^l-1} \text{ s.t. } \mathcal{R}_m(x;y;u)\succeq 0   \end{array}  \right\}.
\]
\end{itemize}

A function $f(X)=(\det X)^{\frac{1}{n}}$ defined over $n\times n$ positive definite matrices $X$ is concave over its domain. It is also possible to show that its hypograph $D^n=\{(X,y)\in\mathbb{S}_{++}^n\times \mathbb{R}_+ : y \leq (\det X)^{\frac{1}{n}}\}$ is semidefinite representable. To confirm, a semidefinite representation of $K^m$ obtained above is used. The next result can be found in Section 4.2 in \cite{ben2001}.
\begin{lemma}
 $D^n=\{(X,y)\in\mathbb{S}_{++}^n\times \mathbb{R}_+ : y \leq (\det X)^{\frac{1}{n}}\}$ is semidefinite representable as
\[D^n\!=\!\left\{ \!\begin{array}{l} (X,y)\in\mathbb{S}_{++}^n\times \mathbb{R}_+ : \\
  \exists \textnormal{ lower triangular matrix } \Delta \in \mathbb{R}^{n\times n} \textnormal{ and } u\in\mathbb{R}^{n_u} \\
  \textnormal{ such that } \mathcal{Q}_n (\textnormal{diag}(\Delta);y;u) \succeq 0  \\
 \textnormal{ and }
\left[\begin{array}{cc}
X & \Delta \\ \Delta^\top & \textnormal{Diag}(\Delta)
\end{array}\right] \succeq 0
 \end{array}  \!\! \right\}\]
where $\textnormal{diag}(\Delta)$ is a vector of diagonal elements of $\Delta$, $\textnormal{Diag}(\Delta)$ is a diagonal matrix whose diagonal elements are the same as those of $\Delta$, and $n_u=2^{\lceil \log_2 n \rceil}-1$.
\end{lemma}
Combining the results obtained so far, it is possible to derive a semidefinite representation of a set such as 
\[
\left\{\begin{array}{l}  (X_1,\cdots, X_T, y)\in \mathcal{S}_{++}^n \times \cdots \times \mathcal{S}_{++}^n \times \mathbb{R}_+ : \\ y \geq \prod_{i=1}^T (\det X_i)^{-\frac{1}{n}} \end{array} \right\}.
\]
 This observation is critical when rewriting a max-det problem (\ref{optprob3hard}) as a standard SDP. Notice that even if the objective function $-\sum_{t=1}^T \log \det \Pi_t $ in (\ref{optprob3hard}) is replaced with $\prod_{i=1}^T (\det X_i)^{-\frac{1}{n}}$, the optimal solution will be unchanged. Thus, the unique optimal solution to (\ref{optprob3hard}) can be also found by solving the following SDP:
\begin{align*}
\min & \;\; y \\
\text{s.t. } & \;\; \Pi_t \succ 0, P_{t|t} \preceq A_{t-1}P_{t-1|t-1}A_{t-1}^\top + W_{t-1} \\
& \;\; \text{Tr}(\Theta_t P_{t|t})\leq D_t,  \forall t=1,\cdots, T \\
& \;\; \left[\begin{array}{cc} P_{t|t}-\Pi_t & P_{t|t}A_t^\top \\ A_tP_{t|t} & A_tP_{t|t}A_t^\top + W_t   \end{array}\right] \succeq 0, \forall t=1,\cdots, T-1 \\
& \;\;  P_{T|T}=\Pi_T, \mathcal{R}_T (z;y;v)\succeq 0 \\
& \;\; \left[\begin{array}{cc}\Pi_t & \Delta_t \\ \Delta_t^\top & \text{Diag}(\Delta_t) \end{array}\right] \succeq 0 \\
& \;\; \mathcal{Q}_n(\text{diag}(\Delta_t);z_t;u_t)\succeq 0, \forall  t=1,\cdots, T.
\end{align*}
Notice that we have introduced additional variables $y\in\mathbb{R}$, $z \in \mathbb{R}^T, v\in \mathbb{R}^{n_v}$ with $n_v=2^{\lceil \log_2 (T+1) \rceil}-1$, $u_1, \cdots, u_T \in \mathbb{R}^{n_u}$ with $n_u=2^{\lceil \log_2 n \rceil}-1$, and lower triangular matrices $\Delta_1, \cdots, \Delta_T \in \mathbb{R}^{n \times n}$.
\fi

\ifdefined\LONGVERSION
\section{Complexity analysis}
\label{seccomplexity}

\subsection{Weighted determinant maximization problem}

Let $X_i \in \mathbb{S}_{++}^{N_i}, i=1,\cdots, n$ be decision variables, and write $X=\text{Diag}(X_1,\cdots, X_n)$. Let $\mathcal{A}$ be a linear operator, and $b \in \mathbb{R}^m$ be a given vector.
In what follows, $\mathcal{I}\subseteq \{1, \cdots, n\}$ is a fixed subset of indices. For every $i\in \mathcal{I}$, let $c_i >0$ be a given constant, while for every  $i\not\in \mathcal{I}$, define $c_i =0$. Symmetric matrix $C$ is also given, and we write $\left<C, X\right>=\text{Tr}(CX)$. For real vectors, $\left<\cdot, \cdot \right>$ means the standard dot product. 
The weighted determinant maximization problem is in the following form:
\begin{align*}
{\bf (P) }\;\; \min_{X_i\in\mathbb{S}^{N_i}} & \;\; \left<C, X\right>-\sum_{i\in \mathcal{I}} c_i \log\det X_i \\
\text{s.t. } & \;\; \mathcal{A}(X)=b,\;\;  X \succeq 0.
\end{align*}
It is an easy exercise to see that the sequential rate-distortion problem (\ref{optprob3}) can be rewritten in this form. The dual problem is
\begin{align*}
{\bf (D) }\;\; \max_{y\in \mathbb{R}^m, \; Z_i\in\mathbb{S}^{N_i}} & \;\; \left<y, b\right>+\sum_{i\in \mathcal{I}} c_i \log\det Z_i + \sum_{i\in \mathcal{I}} c_i N_i \\
\text{s.t. } & \;\; \mathcal{A}^*(y)+Z=C,\;\;  Z \succeq 0.
\end{align*}
where $Z=\text{Diag}(Z_1, \cdots, Z_n)$ and $\mathcal{A}^*$ is the adjoint of $\mathcal{A}$.
We assume that a primal-dual path-following algorithm proposed in  \cite{tsuchiya2007} is used to solve a primal-dual pair of optimization problems {\bf (P)} and {\bf (D)} simultaneously. As in the interior-point methods for SDP, log-determinant barrier functions parameterized by $\nu>0$ is introduced. In particular, we consider the following barrier problems for {\bf (P)} and {\bf (D)} respectively.

\begin{align*}
{\bf (P_\nu) }\;\; \min_{X_i\in\mathbb{S}^{N_i}} & \;\; \left<C, X\right>-\sum_{i=1}^n \max(c_i, \nu) \log\det X_i \\
\text{s.t. } & \;\; \mathcal{A}(X)=b,\;\;  X \succ 0, \\
{\bf (D_\nu) } \! \max_{y\in \mathbb{R}^m, \! Z_i\in\mathbb{S}^{N_i}} & \!\!\!\! \left<y, b\right>\!+\!\!\sum_{i=1}^n \max(c_i,\nu) \log\det Z_i \!+\! \sum_{i\in \mathcal{I}} c_i N_i \\
\text{s.t. } & \;\; \mathcal{A}^*(y)+Z=C,\;\;  Z \succ 0.
\end{align*}

Barrier problems ${\bf (P_\nu) }$ and ${\bf (D_\nu) }$ are designed so that their optimal solutions approaches to those of  ${\bf (P) }$ and ${\bf (D) }$ as $\nu\searrow 0$.
Suppose $X$ and $(y, Z)$ are feasible points for ${\bf (P_\nu)}$ and ${\bf (D_\nu)}$. Let $p_\nu(X)$ and $d_\nu(y,Z)$ be the values of the objective functions in ${\bf (P_\nu)}$ and ${\bf (D_\nu)}$, evaluated at $X$ and $(y,Z)$ respectively. Then, it is easy to show that the quantity $p_\nu(X)- d_\nu(y,Z) $ is minimized when 
\begin{align}
\mathcal{A}(X)=b, X \succ 0,  \mathcal{A}^*(y)+Z=C, Z\succ 0, & \nonumber \\
  X_iZ_i=\max(c_i, \nu)I, & \label{optcond}
\end{align}
and hence, (\ref{optcond}) is the optimality condition for $(X, y, Z)$ to be the primal-dual optimal solution to ${\bf (P_\nu) }$ and ${\bf (D_\nu) }$. The solution of  (\ref{optcond}) parameterized by $\nu$, denoted by $(X(\nu), y(\nu), Z(\nu))$, is called the \emph{primal-dual central path}. The basic idea of the primal-dual path-following algorithm is to trace the trajectory of $(X(\nu), y(\nu), Z(\nu))$ numerically as $\nu$ is reduced to zero.
Every time $\nu$ is updated, $(X, y, Z)$ is updated to $(X+\alpha \Delta X, y+ \alpha \Delta y, Z+ \alpha \Delta Z)$ where $\alpha$ is a step size, so that (\ref{optcond}) is approximately met at the new point. A standard choice of the update direction $(\Delta X, \Delta y, \Delta Z)$ is the scaled Newton direction.
The scaled Newton direction, aiming at approximating the root of nonlinear equations of the form $\mathcal{A}(X)=b, \mathcal{A}^*(y)+Z=C, X_iZ_i=w_i I$, is obtained by solving a linearized equations:
\begin{subequations}
\label{newtondirection}
\begin{align}
\mathcal{A}(\Delta X) &= b-\mathcal{A}(X) \\
\mathcal{A}^*(\Delta y)+\Delta Z&=C-\mathcal{A}^*(y)-Z \\
\mathcal{H}_{P_i}(X_i\Delta Z_i+ \Delta X_i Z_i)&=w_i I - \mathcal{H}_{P_i}(X_i Z_i)
\end{align}
\end{subequations}
where $P_i$ is a nonsingular scaling matrix, and $\mathcal{H}_{P_i}(M)=\frac{1}{2}(P_iMP_i^{-1}+P_i^{-\top}M^\top P^\top)$.
Following \cite{tsuchiya2007}, we assume the choice $P_i=(X_i^{1/2}(X_i^{1/2}Z_iX_i^{1/2}) X_i^{1/2})^{1/2}$ for every $i=1,\cdots, n$, which results in the so-called Nesterov-Todd direction $(\Delta X, \Delta y, \Delta Z)$.

\subsection{Description of the path-following algorithm}
By definition of the primal-dual central path, $\left<X(\nu), Z(\nu)\right>=\sum_{i=1}^n \max(c_i, \nu) N_i$ holds. 
Since the primal-dual optimal solution to ${\bf (P) }$ and ${\bf (D) }$ is approximated by $(X(\nu),Z(\nu)) \rightarrow (X^*,Z^*)$ as $\nu\searrow 0$, it must be that $\left<X^*, Z^*\right>=\sum_{i\in \mathcal{I}} c_i N_i$. 
The path-following algorithm proposed in \cite{tsuchiya2007} sequentially generates  $\{X^k, Z^k\}$ such that $\left<X^k, Z^k\right>\searrow \sum_{i\in \mathcal{I}} c_i N_i$. 

To describe the algorithm more precisely, several notations need to be introduced. Let $\mathcal{F}=\{(X,y,Z): \mathcal{A}(X)=b, X \succ 0, \mathcal{A}^*(y)+Z=C, Z\succ 0\}$ be the primal-dual feasible set. For $(X, y, Z)\in \mathcal{F}$ such that $\left<X, Z\right> > \sum_{i\in \mathcal{I}} c_i N_i$, clearly there exists $\nu> 0$ such that $\left<X, Z\right> = \sum_{i=1}^n \max(c_i, \nu) N_i$.
Using such $\nu$, define the index set $\mathcal{I}^*(X,Z)\subseteq \mathcal{I}$ by $\mathcal{I}^*(X,Z)=\{i : c_i \geq \nu \}$. The \emph{extended normalized duality gap} is the quantity defined by
\[ \mu(X,Z)=\frac{\left<X,Z\right> - \sum_{i\in\mathcal{I}^*(X,Z)} c_i N_i}{ \sum_{i\not\in\mathcal{I}^*(X,Z)}  N_i}. \]
The extended normalized duality gap satisfies the following properties.
\begin{itemize}
\item For $(X,Z)\in \mathcal{F}$ such that $\left<X, Z\right> > \sum_{i\in \mathcal{I}} c_i N_i$, we have $\mu(X,Z)>0$.
\item On the primal-dual central path, we have $\mu(X(\nu),Z(\nu))=\nu$.
\end{itemize}
The path-following algorithm generates a sequence in the neighborhood of the primal-dual central path such that $\mu(X^k, Z^k)\searrow 0$. The neighborhood of the central path is defined explicitly as follows using a constant $\gamma \in (0,1)$ specifying the size of the neighborhood:
\[
\mathcal{N}_\infty(\gamma) = \{(X, y, Z)\in\mathcal{F}: d_\infty(X,Z) \leq \gamma \mu(X,Z) \} \]
where 
\begin{align*} d_\infty(X,Z)= \max \{
& \max_{i\in\mathcal{I}^*(X,Z)} c_i-\lambda_{\text{min}}(X_iZ_i),  \\ & \max_{i\not\in\mathcal{I}^*(X,Z)} \mu(X,Z)-\lambda_{\text{min}}(X_iZ_i) \}
\end{align*}
 is a distance function. It can be shown that if $(X,y,Z)\in \mathcal{N}_\infty(\gamma)$ and $\mu(X,Z)$ is sufficiently small, the duality gap between the original pair of problems ${\bf (P) }$ and ${\bf (D) }$ is upper bounded by $N\mu(X,Z)$, where $N=\sum_{i=1}^n N_i$.
Therefore, if $\{X^k, y^k, Z^k\}$ is a sequence in $\mathcal{N}_\infty(\gamma) $ such that $\mu(X^k,Z^k)\searrow 0$, then any limit point of the sequence are optimal solutions to  ${\bf (P) }$ and ${\bf (D) }$.

\begin{algorithm}[t]
\small
\caption{Primal-dual path following algorithm}
\begin{algorithmic}[1]
\label{pathfollowingalg}
\REQUIRE Parameter $\gamma\in (0,1)$ and initial point $(X^0, y^0, Z^0)\in \mathcal{N}_\infty(\gamma)$ 
\STATE Set $\sigma\in (0,1)$ (rate of barrier update) and $\epsilon >0$ (accuracy of the final solution)
\REPEAT 
\STATE Set $(X,y,Z)=(X^k,y^k,Z^k)$ and compute scaling matrices $P_i, \; i=1,\cdots, n$.
\STATE For $i=1,\cdots,n$, set $w_i=c_i$ if $i\in \mathcal{I}^*(X,Z)$ and $w_i=\sigma\mu(X,Z)$ otherwise. 
\STATE Solve (\ref{newtondirection}) for the scaled Newton direction $(\Delta X, \Delta y, \Delta Z)$.
\STATE Let $\alpha^k$ be the largest $\alpha>0$ such that $(X+\alpha \Delta X, y+\alpha \Delta y, Z+\alpha \Delta Z)\in \mathcal{N}_\infty(\gamma)$.
\STATE Update $(X^{k+1}, y^{k+1}, Z^{k+1})=(X+\alpha^k \Delta X, y+\alpha^k \Delta y, Z+\alpha^k \Delta Z)\in \mathcal{N}_\infty(\gamma)$
\UNTIL{$\mu(X^k, Z^k)\leq \epsilon \mu(X^0, Z^0)$}
\end{algorithmic}
\end{algorithm}

Assuming that the initial point $(X^0, y^0, Z^0)\in \mathcal{N}_\infty(\gamma)$  is already given, the primal-dual path following algorithm proceeds as in Algorithm \ref{pathfollowingalg}. Although the barrier parameter $\nu$ does not explicitly show up, the sequence $(X^{k+1}, y^{k+1}, Z^{k+1})$ approximately follows the central path to the optimal solution. Choosing step sizes as in line 6 makes sense, since it can be shown that the extended normalized duality gap $\mu(X+\alpha \Delta X, Z+\alpha \Delta Z)$ is a monotonically non-increasing function of $\alpha$, and our purpose is to reduce this quantity to zero. Determining such $\alpha^k$ does not have to be precise, so this step is computationally cheap\footnote{What is required is that $\alpha^k$ is large enough in order to guarantee the extended duality gap reduces at some specific rate at every iteration. This turns out to be possible.}. The following result on the iteration complexity is due to  \cite{tsuchiya2007}.

\begin{proposition}
Assuming the Nesterov-Todd direction is used, Algorithm \ref{pathfollowingalg} terminates in $\mathcal{O}(N\log(1/\epsilon)+N)$ iterations.
\end{proposition}

The most computationally expensive step in Algorithm \ref{pathfollowingalg} is the computation of the Newton direction (line 5). The above proposition guarantees that the number of this operation is no more than $\mathcal{O}(N\log(1/\epsilon)+N)$ throughout the path following.

\subsection{Initialization}
Algorithm  \ref{pathfollowingalg} requires an initial point $(X^0,y^0,Z^0)$ in the neighborhood of the central path. 
For instance, let $\nu^0=1$ be the initial barrier parameter, and suppose that the initialization phase is required to give a good approximation of $(X^0,y^0,Z^0)=(X(\nu^0),y(\nu^0),Z(\nu^0))$, so that the path following algorithm can start from this point. Then, the path following phase considered in the previous subsection  involves at most $\mathcal{O}(N\log(1/\epsilon)+N)$ Newton iterations\footnote{Recall that this is the number required to reduce extended normalized duality gap $\mu$ from $1$ to $\epsilon$, and small $\mu$ implies small duality gap for the ${\bf (P) }$-${\bf (D) }$ pair.} until it reaches to an approximated solution $(X, y, Z)$ attaining duality gap less than $\epsilon$ for the primal-dual pair of problems ${\bf (P) }$ and ${\bf (D) }$. 

If a feasible point $X^{\text{feas}}$ for ${\bf (P_{\nu^0})}$ is available, one can use the Newton's method also in the initialization phase to obtain a very good approximation of $(X^0,y^0,Z^0)=(X(\nu^0),y(\nu^0),Z(\nu^0))$. (Dual variables can be also obtained as by-products.)
Moreover, thanks to some desirable properties\footnote{Namely \emph{self-concordance}. See Section 2.2 of \cite{renegar2001}} of the objective function $p_{\nu^0}(X)$, it is also possible to estimate the maximum number of Newton iterations involved in the initialization phase.
For instance, let $U_B$ be an upper bound of the initial level of suboptimality, i.e., $p_{\nu^0}(X^{\text{feas}})-p_{\nu^0}(X^*)\leq U_B$.  Then, starting from $X^{\text{feas}}$, it is possible to obtain an improved solution $X$ such that $p_{\nu^0}(X)-p_{\nu^0}(X^*)\leq \delta$ after at most $11U_B+\log_2\log_2 (1/\delta)$ Newton iterations (Theorem 3 in \cite{vandenberghe1998}).
It is notable that this number does not depend on the problem size $N$ at all.
Moreover, since the doubly logarithmic function grows extremely slowly, very few Newton iterations are needed to obtain an extremely good approximation of $(X^0,y^0,Z^0)=(X(\nu^0),y(\nu^0),Z(\nu^0))$.

Hence, unless it is difficult to find $X^{\text{feas}}$ with moderate $U_B$, it is reasonable to assume that the number $11U_B+\log_2\log_2 (1/\delta)$ is dominated by $\mathcal{O}(N\log(1/\epsilon)+N)$.
Hence, we will use an expression $\mathcal{O}(N\log(1/\epsilon))$ to estimate the number of Newton iterations involved in both initialization and the path following phases.

\subsection{Complexity of SRD problem}
Performing Newton iterations is the main computational burden in Algorithm  \ref{pathfollowingalg} (and initialization). Since it is easy to give an initial feasible point  for problem (\ref{optprob3}), our observation so far indicates that  $\mathcal{O}(N\log(1/\epsilon))$ is a reasonable upper bound on the number of Newton iterations to solve (\ref{optprob3}). Now we will consider how the arithmetic complexity grows as $T$ grows in  (\ref{optprob3}). 

Since $N$ grows proportionally to $T$, the number of Newton iterations required is bounded by  $\mathcal{O}(T\log(1/\epsilon))$. To see how many arithmetic operations are needed to perform a single Newton step, observe that it is possible to write (\ref{optprob3}) in the form of {\bf (P)} in such a way that $\mathcal{A}(\cdot)$ is a banded linear operator, since all linear constraints appearing in  (\ref{optprob3}) contain variables from neighboring time steps. 
Hence, it is shown that the linear system (\ref{newtondirection}) to be solved in each Newton iteration can be written as $Ax=b$, where $A$ is a banded matrix of dimension $\mathcal{O}(T)$ with bandwidth $\mathcal{O}(n)$. (Here, $n$ is the dimension of the Gauss-Markov process (\ref{gmprocess})).
It is well-known that the arithmetic complexity to solve such linear systems is $\mathcal{O}(Tn^2)$. (e.g., \cite{thorson1979}).

To conclude, we have shown that at most $\mathcal{O}(T\log(1/\epsilon))$ Newton iterations are required to obtain an $\epsilon$-optimal solution to  (\ref{optprob3}), and each Newton iteration requires $\mathcal{O}(T)$ arithmetic operations. Thus, we obtain an expression $\mathcal{O}(T^2\log(1/\epsilon))$ for the arithmetic complexity of  (\ref{optprob3}).
\fi

% use section* for acknowledgement
\section*{Acknowledgment}

The authors would like to thank Prof. Sekhar Tatikonda for valuable discussions.

% Can use something like this to put references on a page
% by themselves when using endfloat and the captionsoff option.
\ifCLASSOPTIONcaptionsoff
  \newpage
\fi

% trigger a \newpage just before the given reference
% number - used to balance the columns on the last page
% adjust value as needed - may need to be readjusted if
% the document is modified later
%\IEEEtriggeratref{8}
% The "triggered" command can be changed if desired:
%\IEEEtriggercmd{\enlargethispage{-5in}}

% references section

% can use a bibliography generated by BibTeX as a .bbl file
% BibTeX documentation can be easily obtained at:
% http://www.ctan.org/tex-archive/biblio/bibtex/contrib/doc/
% The IEEEtran BibTeX style support page is at:
% http://www.michaelshell.org/tex/ieeetran/bibtex/
%\bibliographystyle{IEEEtran}
% argument is your BibTeX string definitions and bibliography database(s)
%\bibliography{IEEEabrv,../bib/paper}
%
% <OR> manually copy in the resultant .bbl file
% set second argument of \begin to the number of references
% (used to reserve space for the reference number labels box)

\bibliographystyle{IEEEtran}
\bibliography{InfoRegulator}

\end{document}